\documentclass[11pt,a4paper]{article}

\usepackage{amssymb,amsthm,amsmath}
\usepackage[pdftex]{graphicx}
\usepackage{graphicx}
\usepackage[margin=1in]{geometry}
\usepackage{color}
\usepackage{hyperref}
\usepackage{comment}
\usepackage{bm}
\usepackage{cleveref}
\usepackage[normalem]{ulem}
\usepackage{float}
\usepackage{multicol}
\usepackage{tikz-cd}
\newtheorem{theorem}{Theorem}[section]
\newtheorem{lemma}[theorem]{Lemma}
\newtheorem{proposition}[theorem]{Proposition}

\theoremstyle{definition}

\theoremstyle{remark}

\newtheorem{assume}{Assumption}

\newcommand{\cB}{\mathcal{B}}

\newcommand{\cI}{\mathcal{I}}
\newcommand{\cJ}{\mathcal{J}}

\newcommand{\cM}{\mathcal{M}}

\newcommand{\cP}{\mathcal{P}}
\newcommand{\cQ}{\mathcal{Q}}

\newcommand{\cS}{\mathcal{S}}
\newcommand{\cT}{\mathcal{T}}

\newcommand{\cX}{\mathcal{X}}
\newcommand{\cY}{\mathcal{Y}}
\newcommand{\cZ}{\mathcal{Z}}

\newcommand{\bM}{\mathbf{M}}

\newcommand{\bP}{\mathbf{P}}
\newcommand{\bQ}{\mathbf{Q}}

\newcommand{\balpha}{\bm{\alpha}}

\newcommand{\bc}{\mathbf{c}}

\newcommand{\bv}{\mathbf{v}}

\newcommand{\bx}{\mathbf{x}}
\newcommand{\by}{\mathbf{y}}
\newcommand{\bz}{\mathbf{z}}

\newcommand{\RR}{\mathbb{R}}
\newcommand{\NN}{\mathbb{N}}
\newcommand{\iy}{\mathit{y}}
\newcommand{\biy}{\bm{\mathit{\iy}}}

\newcommand{\ropt}{\rho^\star}

\def\degg#1{\lceil#1\rceil}

\def\pMTX2r{\mathcal{P\!M}(\cT(\cS)_{2r})}
\def\pMQX2r{\mathcal{P\!M}(\cQ(\cS)_{2r})}
\def\pMkTX2r{\mathcal{P\!M}_k(\cT(\cS)_{2r})}
\def\pMkQX2r{\mathcal{P\!M}_k(\cQ(\cS)_{2r})}

\title{Quadratic Optimization over Probability Measures with Coupling Constraints}
\author{Hoang Anh Tran, Yong Sheng Soh}

\begin{document}

\maketitle

\begin{abstract}
    In this paper we consider the task of solving an optimization instance in which the objective is quadratic and where the decision variable is a probability measure. Our class of problems are motivated by applications arising from optimal transport (with the Gromov-Wasserstein problem being a prominent example) and energy landscape minimization. Because the objective depends quadratically on the decision variable, our class of problems fall outside the usual modeling framework of the Generalized Moment Problems (which necessarily requires the objective to be linear). To this end, we propose a hierarchy of convex relaxations that is based on searching over probability measures over products of the base space. These have a natural interpretation with the moment Sum-of-squares (SOS) hierarchy -- a prominent framework for solving polynomial optimization instances, which we adapt to accommodate probability measures. A key conceptual contribution is to introduce a notion of positive-semidefiniteness that extends the usual notion of positive-semidefiniteness over matrices.  

    Under the assumption that the decision variables satisfy certain marginal constraints (as in the Kantorovich formulation of the optimal transport problem), we establish convergence of our hierarchy towards the globally optimal solution.  Under the additional assumption that the objective is a polynomial, we propose a moment--SOS type hierarchy of finite dimensional semidefinite programs whose optimal solution converges to that of the original quadratic optimization over measures. We demonstrate our framework with numerical experiments.

    More generally, the class of optimization instances over measures and where the objective and/or constraint depends on the decision in a polynomial way is a fundamental problem. It is hoped that our work provides a road-map as to how the ideas of the SOS -- ordinarily developed in the context of polynomial optimization -- may be applied to a broader class of non-linear problems involving measures.
\end{abstract}

\section{Introduction}

In this paper, we consider the task of solving an optimization instance in which we seek a probability measure that minimizes a quadratic objective, possibly subject to additional linear equality or inequality constraints.  Formally, let $\cX \subseteq \mathbb{R}^{d}$ be a Borel set, and let $\mathcal{P}(\cX)$ denote the set of probability measures over $\cX$.  Consider the following optimization instance:  
\begin{equation} \label{eq:general_quadform}
\begin{aligned}
\rho^{\star} ~ := ~ \inf_{\pi \in \mathcal{P}(\cX)} \quad & \int_{\cX^2} c(\bx_1,\bx_2) d\pi(\bx_1) d\pi(\bx_2) \\
\mathrm{s.t.} \quad & \int_{\cX} h_{\gamma} (\bx) d\pi(\bx) \leq b_{\gamma} \quad \forall \gamma \in \Gamma
\end{aligned}.
\end{equation}
Here, $c:\cX \times \cX \to \mathbb{R}$ is a measurable cost function, while $h_{\gamma}:\cX\to\mathbb{R}$ and $b_{\gamma}\in\mathbb{R}$ specify linear inequality constraints.  Equality constraints can be modelled within~\eqref{eq:general_quadform} by representing these as a pair of opposite inequalities.  

%The class of problems captured by \eqref{eq:general_quadform} are challenging to solve, in large part because (i) these take place over the space of probability measures, which is infinite dimensional, and (ii) the objective has quadratic dependency on the decision variable in the sense that $\pi$ appears twice in the integral.  

{\bf The Generalized Moment Problem.}  The class of problems we study \eqref{eq:general_quadform} is closely related to a class of problems known as the {\em Generalized Moment Problem (GMP)}.  Formally, the GMP seeks a probability measure that minimizes a linear objective, possibly subject to additional linear inequality constraints: 
\begin{equation} \label{eq:gmp}
\begin{aligned}
\inf_{\pi \in \mathcal{P}(\cX)} \quad & \int_\cX c_0(\bx) d\pi(\bx) \\
\mathrm{s.t.} \quad & \int_\cX h_{\gamma} (\bx) d\pi(\bx) \leq b_{\gamma} \quad \text{for all } \gamma \in \Gamma %\\
%& \mathrm{supp} (\pi) \subset \{ x : g_1(x) \geq 0, \ldots, g_m(x) \geq 0 \}  
\end{aligned}.
\end{equation}
(To be clear, the formulation of the GMP such as in \cite{Las:09} is often stated more generally; for instance, in ~\cite{Las:09}, the decision variable $\pi$ may be extended to finite signed measures.) 

The problem we consider \eqref{eq:general_quadform} falls outside the modeling framework of GMPs as it is described in \eqref{eq:gmp} because the objective has {\em quadratic} dependence on $\pi$ whereas the objective as well as the constraints in \eqref{eq:gmp} all have {\em linear} dependence on $\pi$.  The essence of this paper is to develop techniques that allow us to solve optimization instances that deal with the quadratic dependency.  The basis of our ideas come from the moment-SOS hierarchy, which is a principled framework for solving Polynomial Optimization instances to global optimality (see e.g.,~\cite{BPT:12,Las:01,Las:09,Par:00}).  Our main contribution is to explain how these ideas extend to optimization instances over probability measures.

\subsection{An SOS-type hierarchy adapted for probability measures} \label{sec: contribution}

Our first contribution is to describe a hierarchy of convex relaxations of increasing strength for the problem~\eqref{eq:general_quadform}.  First, observe that the objective in~\eqref{eq:general_quadform} can be expressed as follows:
\begin{displaymath}
   \int_{\cX^2} c(\bx_1,\bx_2)~d\pi(\bx_1)\pi(\bx_2) = \int_{\cX^2} c(\bx_1,\bx_2)~d(\pi\otimes \pi)(\bx_1,\bx_2).
\end{displaymath}
Since $\pi \in \cP(\cX)$ is a probability measure, its product $\pi\otimes \pi$ is also a probability measure.  By introducing a auxiliary decision variable $\bP \in \cP(\cX^2)$, and then subsequently enforcing the constraint $\bP = \pi \otimes \pi$, the problem~\eqref{eq:general_quadform} can be equivalently reformulated as follows:
\begin{equation}\label{eq:general_quadform_lifted}
\begin{aligned}
    \rho^\star
    ~~=~~
    \inf_{\bP\in\cP(\cX^2)} \quad
    & \int_{\cX^2}
        c(\bx_1,\bx_2)\,
        d\bP(\bx_1,\bx_2) \\
    \mathrm{s.t.} \quad
    & \int_{\cX^2}
        h_\gamma(\bx_1)\,
        d\bP(\bx_1,\bx_2)
        \leq b_\gamma,
        \qquad \forall\gamma\in\Gamma, \\
    & \bP=\pi \otimes \pi
        \quad\text{for some }\pi\in\cP(\cX).
\end{aligned}
\end{equation}
Although the objective and marginal constraints in
\eqref{eq:general_quadform_lifted} are now linear in $\bP$, the constraint $\bP = \pi \otimes \pi$ is {\em not} linear and is generally {\em not} convex in $\pi$.

Our basic strategy is to relax the product-measure constraint $\bP = \pi \otimes \pi$ by lifting to the product space.  Formally, let $r\in\mathbb{Z}_{>0}$ be a positive integer.  Consider the following product space:
\begin{equation*}
    \cX^{2r}
    :=
    \underbrace{
        \cX\times\cdots\times\cX
    }_{2r\text{ times}}.
\end{equation*}
We define relaxations in which the decision variables are probability measures over $\cX^{2r}$; i.e., they reside in $\cP(\cX^{2r})$.  These
construction are based on ideas arising from the
moment--SOS hierarchy for polynomial optimization, and we discuss these connections in Section~\ref{sec: connection to POP}.  The relaxations we specify increase in strength as one increases the parameter $r$, but at the expense
of increasing computational cost.

More concretely, let $S_{2r}$ denote the symmetric group on $[2r]$.  Let $\bP\in\cP(\cX^{2r})$ denote our decision variable.  We impose the following constraints in our problem:
\begin{enumerate}
    \item
    The measure $\bP$ is invariant under permutations of its block-coordinates; i.e., for every $\sigma\in S_{2r}$ and every
    $C_1,\ldots,C_{2r}\in\cB(\cX)$, one has
    \begin{equation}\label{eq:cond_sym}\tag{Sym}
        \bP(C_1\times\cdots\times C_{2r})
        =
        \bP(C_{\sigma(1)}\times\cdots\times C_{\sigma(2r)}).
    \end{equation}
    Here, $\cB(\cX)$ denotes the
    Borel $\sigma$-algebra on $\cX$.  In words, $\bP$ is a finitely
    exchangeable probability measure.

    \item We require $\bP_i :=(\mathrm{pr}_i)_{\#}\bP$, the marginal of $\bP$ onto a single $\cX$, to satisfy all linear constraints imposed on $\pi$.  This can also be stated as:
    \begin{equation}\label{eq:cond_mar}\tag{Mar}
        \int_{\cX^{2r}}
            h_\gamma(\bx_1)\,
            d\bP(\bx_1,\ldots,\bx_{2r})
        \leq b_\gamma,
        \qquad \forall\gamma\in\Gamma.
    \end{equation}
    Without loss of generality, we take $i=1$ in the above.  This is valid because of \eqref{eq:cond_sym}.

    \item The measure $\bP$ satisfies the following:
    For all $t\in[r]$, all bounded Borel-measurable functions
    \begin{equation*}
        f:\cX^{r-t}\longrightarrow\mathbb{R},
    \end{equation*}
    and all nonnegative bounded Borel-measurable functions
    \begin{equation*}
        g:\cX^{2t}\longrightarrow\mathbb{R}_+,
    \end{equation*}
    the following integral is non-negative \begin{equation}\label{eq:cond_psd}\tag{PSD}
    \begin{aligned}
        \int_{\cX^{2r}}
        & f(\bx_1,\ldots,\bx_{r-t})\,
          f(\bx_{r-t+1},\ldots,\bx_{2r-2t}) \\
        & {}\times
          g(\bx_{2r-2t+1},\ldots,\bx_{2r})\,
          d\bP(\bx_1,\ldots,\bx_{2r})
        \geq 0.
    \end{aligned}
    \end{equation}
    In the case where $t=0$, $g$ is to be understood as a nonnegative constant.  An alternative definition of \eqref{eq:cond_psd} is as follows:  Suppose one fixes $t$ and $g$.  Define the bilinear form:
    \begin{equation*}
    \begin{aligned}
        \mathsf{B}_{t,g}(f_1,f_2)
        :=
        \int_{\cX^{2r}}
            & f_1(\bx_1,\ldots,\bx_{r-t})\,
              f_2(\bx_{r-t+1},\ldots,\bx_{2r-2t}) \\
            & {}\times
              g(\bx_{2r-2t+1},\ldots,\bx_{2r})\,
              d\bP.
    \end{aligned}
    \end{equation*}
    Then \eqref{eq:cond_psd} is equivalent to requiring $\mathsf{B}_{t,g}$ be a positive semidefinite bilinear form for all $g:\cX^{2t}\rightarrow\mathbb{R}_+$.
\end{enumerate}

We define the $r$-th relaxation by
\begin{equation}\label{eq:general-r-th}
\begin{aligned}
    \rho^{(r)}
    :=
    \inf_{\bP\in\cP(\cX^{2r})} \quad
    & \int_{\cX^{2r}}
        c(\bx_1,\bx_2)\,
        d\bP(\bx_1,\ldots,\bx_{2r}) \\
    \mathrm{s.t.} \quad
    & \bP\text{ satisfies }
      \eqref{eq:cond_sym},
      \eqref{eq:cond_mar},
      \text{ and }\eqref{eq:cond_psd}.
\end{aligned}
\end{equation}

To see that~\eqref{eq:general-r-th} specifies a relaxation, let $\pi$
be feasible for~\eqref{eq:general_quadform} and set
\begin{equation*}
    \bP:=\pi^{\otimes 2r}.
\end{equation*}
One can easily check that \eqref{eq:cond_sym} and \eqref{eq:cond_mar} are satisfied.  The expression in~\eqref{eq:cond_psd} specialized to $\bP:=\pi^{\otimes 2r}$ factorizes as
\begin{equation*}
    \left(
        \int_{\cX^{r-t}}
            f\,d\pi^{\otimes r-t}
    \right)^2
    \left(
        \int_{\cX^{2t}}
            g\,d\pi^{\otimes2t}
    \right)
    \geq 0.
\end{equation*}
Hence \eqref{eq:cond_psd} is also satisfied.  Consequently, every feasible measure $\pi$ induces a feasible measure $\pi^{\otimes 2r}$ for~\eqref{eq:general-r-th}, from which we conclude
\begin{equation*}
    \rho^{(r)}\leq\rho^\star.
\end{equation*}
Furthermore, by marginalizing a feasible measure at level $(r+1)$ onto its first $2r$ coordinates we obtain a feasible measure at level $r$.  It follows that
\begin{equation*}
    \rho^{(r)}
    \leq
    \rho^{(r+1)}
    \leq
    \rho^\star,
    \qquad r\in\mathbb{Z}_{>0}.
\end{equation*}
Thus, the hierarchy provides a monotone sequence of lower bounds on
the optimal value of~\eqref{eq:general_quadform}.

% {\bf Convexity.}  The optimization instance \eqref{eq:general-r-th} can be understood as a convex program over the space of probability measures in the following sense:  Given two probability measures $\mu_0$, $\mu_1$, one can define a convex combination of these measures $\mu_{\theta} = (1-\theta) \mu_0 + \theta \mu_1$ by
% \begin{equation*}
%     \mu_{\theta} (C) = (1-\theta) [\mu_0] (C)+ \theta [\mu_1] (C) \quad \forall \theta \in [0,1], \forall \text{ measurable } C.
% \end{equation*}
% Under this interpretation, the constraint set in \eqref{eq:general-r-th} can be viewed as a convex subset of $\mathcal{P}(\cX^{2r})$.  Moreover, because the objective is linear in $\bP$, the optimization instance \eqref{eq:general-r-th} can be understood as a convex program.  In particular, the purpose of stating \eqref{eq:general-r-th} is that it should be viewed as suitable notion of a convex relaxation of \eqref{eq:general_quadform}.

{\bf Convexity.}  The formulation \eqref{eq:general-r-th} specifies a convex relaxation in the following sense.  Given a pair of measures $\bP_0,\bP_1\in\cP(\cX^{2r})$, and any $\theta \in [0,1]$, we may define a convex combination of these measures
\begin{equation*}
    \bP_\theta
    :=
    (1-\theta)\bP_0+\theta\bP_1
\end{equation*}
as the probability measure satisfying
\begin{equation*}
    \bP_\theta(C)
    =
    (1-\theta)\bP_0(C)
    +
    \theta\bP_1(C),
    \qquad
    \forall C\in\cB(\cX^{2r}).
\end{equation*}
By linearity, 
\eqref{eq:cond_sym} and
\eqref{eq:cond_mar} are preserved under convex combinations.   Suppose that $\bP_0$ and $\bP_1$ are measures that satisfy~\eqref{eq:cond_psd}.  Then, for every admissible choice of $t$, $f$, and $g$, the left-hand side of~\eqref{eq:cond_psd} can be written as
\begin{equation*}
\begin{aligned}
    &\int_{\cX^{2r}}
        f(\bx_1,\ldots,\bx_{r-t})
        f(\bx_{r-t+1},\ldots,\bx_{2r-2t})
        g(\bx_{2r-2t+1},\ldots,\bx_{2r})
        \,d\bP_\theta \\
    &\qquad =
        (1-\theta)
        \int_{\cX^{2r}} f f g\,d\bP_0
        +
        \theta
        \int_{\cX^{2r}} f f g\,d\bP_1
        \geq 0.
\end{aligned}
\end{equation*}
Combining these, we conclude that the feasible region of~\eqref{eq:general-r-th} is a convex subset of $\cP(\cX^{2r})$.  Finally, we note that the objective is linear in $\bP$ since
\begin{equation*}
    \int_{\cX^{2r}}c(\bx_1,\bx_2)\,d\bP_\theta
    =
    (1-\theta)
    \int_{\cX^{2r}}c(\bx_1,\bx_2)\,d\bP_0 + \theta
    \int_{\cX^{2r}}c(\bx_1,\bx_2)\,d\bP_1.
\end{equation*}
Thus problem~\eqref{eq:general-r-th} can be understood as a convex program over probability measures. %This convexity, together with the inclusion of every product measure $\pi^{\otimes 2r}$ induced by a feasible solution of \eqref{eq:general_quadform}, justifies interpreting \eqref{eq:general-r-th} as a convex relaxation of \eqref{eq:general_quadform}.

{\bf Certifying global optimality.}
As shown later in Proposition~\ref{thm:simpleLB}, the optimal value
$\rho^{(r)}$ of the $r$-th relaxation provides a lower bound on
$\rho^\star$. Let $\widetilde{\bP}\in\cP(\cX^{2r})$ be feasible for
\eqref{eq:general-r-th}, and let
\begin{equation*}
    \widetilde{\pi}
    :=\widetilde{\bP}\big|_{\cX}
\end{equation*}
denote its first marginal. By the symmetry condition
\eqref{eq:cond_sym}, this definition is independent of the coordinate
chosen. Moreover, the marginal constraints~\eqref{eq:cond_mar} ensure
that $\widetilde{\pi}$ is feasible for~\eqref{eq:general_quadform}.  Let
\begin{equation*}
    \operatorname{OBJ}(\widetilde{\pi})
    :=
    \int_{\cX^2}
        c(\bx_1,\bx_2)\,
        d\widetilde{\pi}(\bx_1)\,
        d\widetilde{\pi}(\bx_2)
\end{equation*}
denote the objective value of $\widetilde{\pi}$ in
\eqref{eq:general_quadform}. We then obtain the following inequalities:
\begin{equation}\label{eq:optimality_certificate}
    \rho^{(r)}
    \leq
    \rho^\star
    \leq
    \operatorname{OBJ}(\widetilde{\pi}).
\end{equation}
Consequently, if for some $r$ we obtain
\begin{equation*}
    \rho^{(r)}
    =
    \operatorname{OBJ}(\widetilde{\pi}),
\end{equation*}
then both inequalities in~\eqref{eq:optimality_certificate} are equalities.  It follows that
\begin{equation*}
    \rho^{(r)}=\rho^\star
\end{equation*}
and that $\widetilde{\pi}$ is a globally optimal solution of
\eqref{eq:general_quadform}.  In particular, in addition to being an {\em a posteriori} certificate of global optimality, the quantity
\begin{equation*}
    \operatorname{OBJ}(\widetilde{\pi})-\rho^{(r)}
\end{equation*}
provides an upper bound on the optimality gap.  

Regardless of whether these equalities hold, the marginal $\widetilde{\pi}$ (obtained after computing $\widetilde{\bP}$) may be applied as a candidate solution of the original problem \eqref{eq:general_quadform}.  Following earlier arguments, $\widetilde{\pi}$ would in fact be optimal for \eqref{eq:general_quadform} if the following held
\begin{equation*}
    \widetilde{\bP} = \widetilde{\pi}^{\otimes 2r}.
\end{equation*}
%Many of the ideas we state are standard; see, for instance~\cite{burer:2024} and references therein.

\if0
A sufficient condition for the resulting certificate to be
exact is that the lifted optimizer has the product form
\begin{equation*}
    \widetilde{\bP}
    =
    \widetilde{\pi}^{\otimes 2r}
    \qquad
    \text{for some feasible }
    \widetilde{\pi}\in\cP(\cX).
\end{equation*}
Indeed, in this case,
\begin{equation*}
\begin{aligned}
    \rho^{(r)}
    &=
    \int_{\cX^{2r}}
        c(\bx_1,\bx_2)\,
        d\widetilde{\pi}^{\otimes 2r}
            (\bx_1,\ldots,\bx_{2r}) \\
    &=
    \int_{\cX^2}
        c(\bx_1,\bx_2)\,
        d\widetilde{\pi}(\bx_1)\,
        d\widetilde{\pi}(\bx_2)
    =
    \operatorname{OBJ}(\widetilde{\pi}).
\end{aligned}
\end{equation*}
\fi

\subsection{Optimization over coupling measures}

One of our main goals is to characterize the settings under which $\rho^{(r)} \rightarrow \rho^\star$ as $r \rightarrow \infty$.  At this juncture, we only have partial understanding of when this happens.  A large part of the difficulty is that the problem formulation as it is stated in~\eqref{eq:general_quadform} describes a very broad and permissive set of constraints.  

In the remainder of this paper, we will pursue a less ambitious agenda.  To this end, we consider a specific family of constraints that are grounded by applications in optimal transport.  Our second contribution is to establish convergence of $\rho^{(r)}$ towards $\rho^\star$ under these conditions.  

More formally, let $m \geq 2$, and let $\{ \cX_{i} \}_{i=1}^{m}$, $\cX_{i} \subset \mathbb{R}^{n_i}$, be a collection of Borel subsets.  For each $i \in [m]$, let $\mu_i \in \mathcal{P}(\cX_{i})$ be a probability measure supported over the respective spaces $\cX_{i}$.  We take $\cX$ in our problem formulation to be the product space
\begin{displaymath}
    \cX := \cX_1 \times \cdots \times \cX_m \in \RR^n,\quad n:= n_1+\ldots+n_m.
\end{displaymath}
For each $i\in[m]$, let
\begin{displaymath}
    \mathrm{pr}_i:\cX\longrightarrow\cX_i
\end{displaymath}
denote the canonical projection onto the $i$-th coordinate. We define the set of couplings of $\mu_1,\ldots,\mu_m$ by
\begin{equation}\label{eq:coupling_defn}
    \Pi(\mu_1,\ldots,\mu_m)
    :=
    \left\{
        \pi\in\cP(\cX):
        \pi\big|_{\cX_i}=\mu_i
        \quad \forall i\in[m]
    \right\}.
\end{equation}
In the standard definition, $(\mathrm{pr}_i)_{\#}\pi$ denotes the push-forward of $\pi$ under $\mathrm{pr}_i$.  For simplicity, we simplify the notation of push-forward measure onto the marginal spaces by $\pi\big|_{\cX_i} = (\mathrm{pr}_i)_{\#}\pi$.  Note that the marginal constraints in~\eqref{eq:coupling_defn} can be equivalently expressed within the formulation of \eqref{eq:general_quadform} via:
\begin{displaymath}
    \int_{\cX}
        \varphi_i\bigl(\mathrm{pr}_i(\bx)\bigr)\,d\pi(\bx)
    =
    \int_{\cX_i}
        \varphi_i(\bx_i)\,d\mu_i(\bx_i)
\end{displaymath}
for every $i\in[m]$ and every bounded Borel-measurable function
$\varphi_i:\cX_i\to\mathbb{R}$. The elements of
$\Pi(\mu_1,\ldots,\mu_m)$ are called \emph{couplings}, or
\emph{transport plans}, of the prescribed marginals
$\mu_1,\ldots,\mu_m$. Notice that this set is nonempty, since
\begin{displaymath}
    \mu_1\otimes\cdots\otimes\mu_m
    \in\Pi(\mu_1,\ldots,\mu_m).
\end{displaymath}
In the case where $m=2$, the definition \eqref{eq:coupling_defn} recovers the classical definition of a coupling as it appears in the Kantorovich formulation as it is understood in optimal transport \cite{San:15,Vil:03,Vil:08}.  The definition in \eqref{eq:coupling_defn} extends more generally and is applicable to, for instance, multi-marginal extensions of the classical optimal transport problem~\cite{Pass:15}.

%Let $\{ (\cX_{i},d_i) \}_{i=1}^{m}$ be a collection of metric spaces.  For each of these, let $\mu_i$ be a probability measure over the respective metric spaces $(\cX_{i},d_i)$.  Let $\CX = \CX_1 \times \dots \times \CX_m$ denote the product space.  
%Given $\CX = \CX_1 \times \dots \times \CX_m$, and each metric space $(\CX_i,d_i)$ is equipped with a probability measure $\mu_i$. Let $n_i$ be the dimension of of $\CX_i$, and define the set of coupling measure 

In the remainder of this paper, we consider the following class of optimization instances: 
\begin{equation} \label{eq:ot_quadform}
    \ropt ~:=~\; \inf \int_{\cX \times \cX} c(\bx_1,\bx_2) \ d(\pi \otimes \pi)(\bx_1,\bx_2) \quad \mbox{s.t.} \quad \pi \in \Pi(\mu_1,\ldots,\mu_m).
\end{equation}
Here, $\bx_1 = (\bx_{11},\dots,\bx_{1m}), \bx_2 = (\bx_{21},\dots,\bx_{2m}) \in \cX_1 \times \dots \cX_m$.

\subsection{The definition of ``solve''}

To state our third contribution, it is necessary to clarify the notion in which one ``solves'' an optimization instance such as \eqref{eq:general_quadform}.  First, one typically makes a distinction between computing the optimal {\em value} to a problem, as compared to as output an optimal {\em solution}.  Suppose we consider the latter situation.  Then there is an additional subtlety in that probability measures are typically infinite dimensional, and it unclear how a generic optimal solution is described.

To this end, there are broadly two distinct paradigms we are aware of.  The first of these is based on the observation that -- under suitable assumptions -- a probability measure is uniquely characterized by its full sequence of moments.  Although the complete sequence is infinite, it is countable and can be approximated by a finite sequence obtained via a truncation at some appropriate degree.  We call methods that solve for optimization instances such as \eqref{eq:general_quadform} by parameterizing the problem in terms of sequences of moments the \emph{moment-sequence paradigm}. 

The second of this is based on the observation that a probability measure can be approximated by a finite atomic measure
\begin{displaymath}
    \pi_N
    :=
    \sum_{i=1}^{N} w_i\delta_{\bx_i},
    \qquad
    w_i\geq0,
    \qquad
    \sum_{i=1}^{N}w_i=1.
\end{displaymath}
Given a generic measure, a suitable approximation may be obtained either by discretizing the underlying spaces or by drawing samples from the individual measures. The relevant notion of closeness is typically weak convergence.  We call methods based on these ideas the \emph{empirical-measure paradigm}.

One of our reasons for stating formulation~\eqref{eq:general-r-th} is that it retains flexibility to accommodate both paradigms.  Nevertheless, there may be settings in which one favors solutions based on the moment--sequence paradigm.  (As a concrete example, if one wishes to compute the optimal value independently of the errors associated with approximating a continuous measure with an empirical measure, then moment-sequence-based methods are certainly preferable.)  To this end, our third contribution is to develop a framework for solving \eqref{eq:ot_quadform} via the moment-sequence paradigm building on the ideas of the moment-SOS hierarchy.  A key technical challenge is to translate the requirement \eqref{eq:cond_sym} to a form that is amenable to a SOS representation.  We describe these ideas in Section~\ref{sec: moment_hierarchy}.

\subsection{Examples} \label{sec:examples}

We discuss examples of problems that fit into the framework \eqref{eq:ot_quadform} described.

{\bf Optimal Transport.}  The simplest example of a problem that fits into \eqref{eq:ot_quadform} is the Kantorovich formulation of the classical optimal transport:  Let $m=2$, and let
\begin{displaymath}
    \cX=\cY\times\cZ,
    \qquad
    \cY=\cX_1,
    \qquad
    \cZ=\cX_2.
\end{displaymath}
Let $\mu_{\cY}\in\cP(\cY)$ and $\mu_{\cZ}\in\cP(\cZ)$ be probability measures over $\cY$ and $\cZ$ respectively.  Then $\Pi(\mu_\cY,\mu_\cZ)$ is the set of all feasible couplings satisfying the marginal constraints.  The Kantorovich formulation of the OT problem seeks the optimal coupling that minimizes the cost of moving mass from one distribution to another so as to minimize cost, and this can be expressed as the following
\begin{equation*} \label{eq:kantorovich_ot}
    \min_{\pi} ~~  \int_{\cY \times \cZ} c(\by,\bz) \, d\pi(\by,\bz)  \quad
    \mathrm{s.t.} \quad \pi \in \Pi(\mu_\cY,\mu_\cZ).
\end{equation*}
Here, $c(\by,\bz)$ denotes the cost of moving a unit mass from location $\by \in \mathcal{Y}$ to location $\bz \in \mathcal{Z}$.  Note in this particular instance that the objective has {\em linear} dependency, rather than {\em quadratic}, on the decision variable $\pi$.

% {\bf Gromov-Wasserstein.}  The Gromov-Wasserstein problem is an extension of the classical Optimal Transport problem to problem settings where a cost function ascribing the price of moving mass ($c$, as in the above) from one location to another may be absent \cite{Mem:07,Mem:11,Sturm:06}.  To circumvent this, the Gromov-Wasserstein problem supposes that the source and target spaces may be modeled as metric spaces, and on which on seeks couplings that minimize the {\em distortion} of the metric spaces.  

% Formally, a metric measure (mm) space $(\cX,d_{\cX},\mu_{\cX})$ is a triple where $(\cX,d_{\cX})$ specifies a metric space, and where $\mu_{\cX}$ denotes a measure over $\cX$.  Suppose $(\cY,d_{\cY},\mu_{\cY})$ and $(\cZ,d_{\cZ},\mu_{\cZ})$ are a pair of mm-spaces.  The Gromov-Wasserstein problem seeks the minimal solution to the following:
% \begin{equation*} \label{eq:gw}
%     \min_{p} ~~ \int_{(\mathcal{Y} \times \mathcal{Z})^2} c(\by_0,\bz_0,\by_1,\bz_1) \, d\pi(\by_0,\bz_0) \, d\pi(\by_1,\bz_1) \quad 
%     \mathrm{s.t.} \quad \pi \in \Pi.
% \end{equation*}
% Here, the cost $c(\by_0,\bz_0,\by_1,\bz_1)$ is often modelled as the discrepancy between the metric evaluations between pairs of points in $\mathcal{Y}$ and in $\mathcal{Z}$
% \begin{equation*}
%     c(\by_0,\bz_0,\by_1,\bz_1) := \mathrm{Loss} (d_{\mathcal{Y}} (\by_0,\by_1) , d_{\mathcal{Z}} (\bz_0,\bz_1) ).
% \end{equation*}
% As seen above, it can be modeled as an instance of
% \eqref{eq:ot_quadform}.

{\bf Gromov--Wasserstein.}
The Gromov--Wasserstein problem extends classical optimal transport to settings in which a cost function that directly compares points across the source and target spaces may be unavailable. Instead, it compares the intrinsic geometric structures of the two spaces and seeks a coupling that minimizes their metric distortion; see, for instance,~\cite{Mem:07,Mem:11,Sturm:06}.

Formally, a metric measure space, abbreviated as an mm-space, is a triple
\begin{displaymath}
    (\cX,d_{\cX},\mu_{\cX}),
\end{displaymath}
where $(\cX,d_{\cX})$ is a metric space and $\mu_{\cX}$ is a Borel probability measure on $\cX$. Consider two mm-spaces
\begin{displaymath}
    (\cY,d_{\cY},\mu_{\cY})
    \qquad\text{and}\qquad
    (\cZ,d_{\cZ},\mu_{\cZ}).
\end{displaymath}
Given a measurable loss function
\begin{displaymath}
    \mathrm{Loss}:\mathbb{R}_{+}\times\mathbb{R}_{+}
    \longrightarrow\mathbb{R}_{+},
\end{displaymath}
the associated Gromov--Wasserstein problem is
\begin{equation}\label{eq:gw}
    \inf_{\pi\in\Pi(\mu_{\cY},\mu_{\cZ})}
    \int_{(\cY\times\cZ)^2}
        c(\by_0,\bz_0,\by_1,\bz_1)\,
        d\pi(\by_0,\bz_0)\,
        d\pi(\by_1,\bz_1),
\end{equation}
where
\begin{displaymath}
    c(\by_0,\bz_0,\by_1,\bz_1)
    :=
    \mathrm{Loss}
    \left(
        d_{\cY}(\by_0,\by_1),
        d_{\cZ}(\bz_0,\bz_1)
    \right).
\end{displaymath}
Thus, rather than directly comparing a point $\by\in\cY$ with a point $\bz\in\cZ$, the Gromov--Wasserstein objective compares the distances between pairs of points within the respective spaces. A common choice is
\begin{displaymath}
    \mathrm{Loss}(a,b)=|a-b|^p,
    \qquad p\geq 1,
\end{displaymath}
when the corresponding $p$-Gromov--Wasserstein distance is obtained by taking the $p$-th root of the optimal value in~\eqref{eq:gw}, up to a possible normalization factor depending on the convention adopted.
In particular, the quadratic-loss formulation corresponds to
\begin{displaymath}
    c(\by_0,\bz_0,\by_1,\bz_1)
    =
    \left|
        d_{\cY}(\by_0,\by_1)
        -
        d_{\cZ}(\bz_0,\bz_1)
    \right|^2.
\end{displaymath}

Since the coupling $\pi$ appears twice in the integral, problem~\eqref{eq:gw} depends quadratically on its decision variable. Consequently, the Gromov--Wasserstein problem is a direct instance of~\eqref{eq:ot_quadform}, with
\begin{displaymath}
    \cX=\cY\times\cZ
    \qquad\text{and}\qquad
    \Pi(\mu_{\cY},\mu_{\cZ}).
\end{displaymath}
Unlike the Kantorovich problem~\eqref{eq:kantorovich_ot}, this quadratic dependence generally makes the Gromov--Wasserstein problem nonconvex.

\subsection{Related work}

We briefly discuss relevant prior work.

{\bf Quadratic optimization over probability simplex and SOS-lifts.}  Suppose that $\cX$ is supported on a finite collection of points.  For the sake of this discussion, we ignore the presence of the affine constraints (i.e., suppose $\Gamma = \emptyset$).  Then \eqref{eq:general_quadform} can be expressed as the following finite dimensional quadratic program
\begin{equation} \label{eq:quadopt_simplex}
    \min_{\bx \in \mathbb{R}^{n} } \quad \bx^{\top} C \bx \qquad \mathrm{s.t.} \qquad 
    \bx \in \Delta^{n-1}.
\end{equation}
The problem of optimizing quadratic forms over the probability simplex is a prominent problem -- an excellent reference on this topic can be found here \cite{BdK:02}.  The class of problems described in~\eqref{eq:quadopt_simplex} is NP--hard as it contains the maximum stable set problem (see e.g.,~\cite{Motzkin:65}).  Nevertheless, because~\eqref{eq:quadopt_simplex} can be modeled as a polynomial optimization instance, the ideas of the Sum-of-Squares hierarchy are applicable and may be used to derive a series of semidefinite relaxations that increase in strength and whose solutions increasingly approximate that of \eqref{eq:quadopt_simplex} (the associated convergence rates are studied in ~\cite{Nie:2014,slot:2022,tran:2025convergence,tran:2026convergence}). In a sense, our objective is to generalize these ideas to probability measures, though it is not entirely apparent how this should be done.

We should emphasize, there is freedom towards how one models \eqref{eq:quadopt_simplex} as a polynomial optimization instance.  Because $\bx \geq 0$, an alternative parameterization is a lift of these variables as squares $x_i = y_i^2$, in which case the optimization instance can be modeled as one over the unit-sphere:
\begin{equation} \label{eq:quadopt_unitsphere}
    \min_{\by \in \mathbb{R}^{n}} \quad (\by \odot \by)^{\top} C (\by \odot \by) \qquad \mathrm{s.t.} \qquad \by \in S^{n-1}.
\end{equation}
Subsequently, one approach to solving \eqref{eq:quadopt_unitsphere} is to seek SOS representations of the polynomial $(\sum_{ij} C_{ij} y_i y_j) (\sum_i y_i^2)^r$~\cite{Nie:2013Approximation,slot:2022}.  It is not apparent, however, if the ideas of lifting $\bx \mapsto \by \odot \by$ as a means of solving \eqref{eq:quadopt_simplex} generalizes to our set-up when the decision variables are probability measures.

{\bf Gromov-Wasserstein.}  The Gromov-Wasserstein (GW) problem is perhaps the most prominent instance of~\eqref{eq:ot_quadform} we are aware of.  Recently, there has been work towards developing semidefinite relaxations for the discretized instance of the GW problem ~\cite{chen:2023semidefinite,tran:2026sum,Vil:2016}.  These ideas were subsequently extended to a full SOS hierarchy in~\cite{tran:2026sum}.  These works focus on discrete measures; in contrast, our work focuses on interpreting these ideas to general measures. 

A different but related work is that of~\cite{MN:24}, which applies the moment--SOS hierarchy to a range of problems arising in optimal transport, of which one such example is the GW problem.  Under the assumption that the cost $c$ is polynomial, the work in~\cite{MN:24} proposes a framework for solving GW instances via the standard moment--SOS hierarchy.  To be clear, the GW objective has quadratic dependency on the measure $\pi$, and as such is ordinarily out of reach of the usual moment--SOS hierarchy.  To this end, the authors in~\cite{MN:24} propose an alternating minimization framework in which one of these factors is treated as constant while one minimizes over the other.  Because the manner in which the hierarchy is applied is atypical, the usual convergence guarantees associated with the moment--SOS do not apply; the work in~\cite{MN:24} does not provide alternative guarantees tailored their method as well.

In contrast to~\cite{MN:24}, our approach to dealing with the quadratic dependency on $\pi$ is to apply a {\em lift}.  A key advantage is that this allows us to state a moment--SOS type hierarchy with accompanying convergence guarantees, which is ordinarily not possible via methods based on alternating minimization (see Section \ref{sec: moment_hierarchy}).

{\bf de Finetti theorems.}  An infinite sequence of random variables $\{ X_1,X_2,\ldots \}$ is said to be {\em exchangeable} if, for any $r$ and any permutation $\sigma : [r] \rightarrow [r]$, the tuples $(X_1,\ldots,X_r)$ and $(X_{\sigma(1)},\ldots,X_{\sigma(r)})$ have the same distribution.   The simplest form of the de Finetti theorem (see e.g.,~\cite{deFinetti:37,HewittSavage:55}) states that any infinite sequence of exchangeable random variables taking values in $\{0,1\}$ can be expressed as the mixture of i.i.d. Bernoulli random variables.

On the surface, the de Finetti theorem seems to bear resemblance to our set-up; in particular, the notion of exchangeability appears superficially related to the symmetric property \eqref{eq:cond_sym}.  However, this is also where the similarity ends.  In our set-up, we perform optimization with the goal of recovering a probability measure $\bP = \pi \otimes \pi$ that is (almost surely) the product of a single measure $\pi$.  The de Finetti theorems state that exchangeable random variables are mixture of i.i.d. random variables, which is not relevant to our set-up.  In particular, it is not apparent if the conclusions of the de Finetti theorems have any relevance to our setting.

\subsection*{Notation}
We use $\bx=(x_1,\dots,x_n)$ to denote a vector of variables and $\RR[\bx]$ as the ring of polynomials in $\bx$. Let $\alpha =(\alpha_1,\dots,\alpha_n)$ be a multi-index with length $|\alpha| = \sum_{i=1}^n \alpha_i$. The set of multi-index of length at most $r$ is denoted by $\NN^n_r = \{\alpha \in \NN^n: \ |\alpha| \leq r\}$. We let $\overline{\NN}^n_{r}$ to be the subset of $\NN^n_r$ whose elements have length exactly $r$. The monomials in $\bx$ are written in the form $\bx^{\alpha}= x_1^{\alpha_1}\cdots x_n^{\alpha_n}$. For any polynomial $f(\bx) = \sum_{\alpha}f_{\alpha}\bx^{\alpha} \in \RR[\bx]$, we define the norm $\|f\| = \sqrt{\sum_{\alpha}f_{\alpha}^2}$, and  $\lceil f \rceil := \lceil (\mbox{deg}\ f) /2 \rceil$.

For a fixed measure $\mu$ and a postive number $p$, we denote the $L^p(\mu)$-- norm by $\|\cdot\|_{L^p(\mu)}$, and the set of $L^p(\mu)$--measurable functions by $L^p(\mu)$.

\section{Background on the (moment)-SOS hierarchy} \label{sec:background}

In this section, we provide relevant background on the (moment)-SOS hierarchy necessary to describe our ideas.  A further development of these ideas can be found, for instance, in \cite{Las:09}.  

Broadly speaking, the (moment)-SOS hierarchy is a principled approach for constructing a hierarchy of semidefinite relaxations of increasing strength for solving polynomial optimization instances~\cite{Las:01,Las:09,Par:00}.  More formally, consider the following polynomial optimization instance in which we minimize a polynomial $f \in \RR[\bx]$ over a semialgebraic set:
\begin{equation}\label{POP}
    \min \quad \{ f(\bx) :  \bx \in \RR^n, \  
    g_j(\bx) \geq 0 \ \forall j \in [m], \
    h_i(\bx) = 0 \ \forall i\in [p] \}.
\end{equation}
Here, $f, g_j, h_i \in \RR[\bx]$ are polynomials.  In the following, we will let $\cS$ denote the constraint set:
\begin{equation} \label{eq:domain_S}
    \cS := \{ \bx \in \RR^n: \ 
    g_j(\bx) \geq 0 \ \forall j \in [m],\;
    h_i(\bx) = 0 \ \forall i\in [p] \}.
\end{equation}

\subsection{SOS hierarchy}

The SOS hierarchy starts by re-formulating~\eqref{POP} as the following maximization instance over the cone of non-negative polynomials:
\begin{equation*}
    f^* \;:=\; \max \;
    \{\lambda \in \RR \; : \; f(\bx)-\lambda \in \cP_+(\cS) \}.
\end{equation*}
Here, $\cP_+(\cS)$ denotes the cone of nonnegative polynomials on $\cS$.  Unfortunately, checking membership in $\cP_+(\cS)$ is NP-hard. % as it contains combinatorial optimization instances such as the MAX-CUT problem (see, e.g.,~\cite{Mich:79}).  
To circumvent these difficulties, the SOS hierarchy constructs a sequence of cones via SOS polynomials that form inner approximations of $\cP_+(\cS)$, and whose membership in these cones can be certified via semidefinite programming (SDPs) instances.

%Depending on the defining polynomials of the underlying set $\cS$, either the preordering or the quadratic module associated with $\cS$ is used to approximate $\cP_+(\cS)$. 

More formally, the {\em preordering} $\cT(\cS)$ and the {\em quadratic module} $\cQ(\cS)$ associated with the semialgebraic set $\cS$ are the cones defined as follows:
\begin{equation*}
\begin{aligned}
    \cT(\cS)
    =&
    \bigg \{
     \sum_{J\subseteq[m]}\sigma_J(\bx)g_J(\bx)
    + \sum_{i=1}^p\tau_i(\bx)h_i(\bx):
    \sigma_J\in
    \Sigma[\bx]
    \ \forall J\subseteq[m],
    \tau_i\in
    \RR[\bx]
    \ \forall i\in[p]
    \bigg \}, \\
    \cQ(\cS)
    =&
    \bigg \{
     \sigma_0(\bx)
    +
    \sum_{j=1}^m\sigma_j(\bx)g_j(\bx)
    +
    \sum_{i=1}^p\tau_i(\bx)h_i(\bx): 
    \sigma_0,\sigma_j\in
    \Sigma[\bx]
    \ \forall j\in[m],
    \tau_i\in
    \RR[\bx]
    \ \forall i\in[p]
    \bigg \}.
\end{aligned}
\end{equation*}
Here, we use the notation $g_J := \prod_{j\in J}g_j$ whenever $J \subseteq [m]$, with the convention that $g_{\emptyset} := 1$.  Generally speaking, the preordering and the quadratic module specify independent SOS hierarchies, each of which has its own set of merits and limitations.  We discuss both of these in parallel.

The cones $\cT(\cS)$ and $\cQ(\cS)$ are infinite dimensional.  In practical computations, it is necessary to work with finite dimensional analogues of these.  Subsequently, the {\em truncated preordering} $\cT(\cS)_{2r}$ and the {\em truncated quadratic module} $\cQ(\cS)_{2r}$ associated with the set $\cS$ are simply $\cT(\cS)$ and $\cQ(\cS)$ specified an additional degree bound on the monomials:
\begin{eqnarray*}
    \cT(\cS)_{2r}
    &=&
    \bigg \{
    %q(\bx)
    %=
    \sum_{J\subseteq[m]}\sigma_J(\bx)g_J(\bx)
    +
    \sum_{i=1}^p\tau_i(\bx)h_i(\bx):
    \\
    &&
    \quad
    \sigma_J\in
    \Sigma[\bx]_{2(r-\lceil g_J\rceil)}
    \ \forall J\subseteq[m],\quad
    \tau_i\in
    \RR[\bx]_{2(r-\lceil h_i\rceil)}
    \ \forall i\in[p]
    \bigg \},
    \nonumber\\
    \cQ(\cS)_{2r}
    &=&
    \bigg \{
    %q(\bx)
    %=
    \sigma_0(\bx)
    +
    \sum_{j=1}^m\sigma_j(\bx)g_j(\bx)
    +
    \sum_{i=1}^p\tau_i(\bx)h_i(\bx):
    \\
    &&
    \quad
    \sigma_0\in\Sigma[\bx]_{2r},\quad
    \sigma_j\in
    \Sigma[\bx]_{2(r-\lceil g_j\rceil)}
    \ \forall j\in[m],\quad
    \tau_i\in
    \RR[\bx]_{2(r-\lceil h_i\rceil)}
    \ \forall i\in[p]
    \bigg \}.
\end{eqnarray*}
Here, we impose $\degg{g_{J}} \leq r$ for all $J\subseteq[m]$, $\degg{g_j}\leq r$ for all $j\in[m]$, and $\degg{h_i}\leq r$ for all $i\in[p]$. 

It is clear from these definitions that one has the following hierarchy of inclusions:
\begin{eqnarray*}
    \cT(\cS)_{2}
    \subseteq
    \cT(\cS)_{4}
    \subseteq
    \cdots
    \subseteq
    \cT(\cS)
    \subseteq
    \cP_+(\cS),\\
    \cQ(\cS)_{2}
    \subseteq
    \cQ(\cS)_{4}
    \subseteq
    \cdots
    \subseteq
    \cQ(\cS)
    \subseteq
    \cP_+(\cS).
\end{eqnarray*}
In particular, these cones specify a hierarchy of inner approximations of $\cP_+(\cS)$.  Subsequently, we may define a hierarchy of lower bounds for $f^*$ via the following:
\begin{equation} \label{eq:SOS}
\begin{aligned}
    f^{(r)}_{\rm SOS, PO}
    & \;:=\;
    \sup \;
    \big\{\lambda :
    f(\bx)-\lambda\in\cT(\cS)_{2r}\big\}, \\
    f^{(r)}_{\rm SOS, QM}
    &  \;:=\;
    \sup \;
    \big\{\lambda :
    f(\bx)-\lambda \in \cQ(\cS)_{2r} \big\}.
\end{aligned}
\end{equation}

Notice that $f(\bx) - \lambda > 0$ over $\cS$ whenever $\lambda < f^*$.  One may in principle conclude that the hierarchy $\{f^{(r})\}_{r=1}^{\infty}$ converges to $f^*$ if one is further able to guarantee that a polynomial positive over $\cS$ is eventually inside $\cT(\cS)_{2r}$ and $\cQ(\cS)_{2r}$, for $r$ sufficiently large.  Results of such a nature do in fact exist, and are known as Positivstellens\"atze in the literature:
\begin{theorem}[Schm\"udgen Positivstellensatz] (\cite[pp.~283--313]{Schmud:2017})
    \label{thm: Schmüdgen}
    Suppose $\cS$ as specified in \eqref{eq:domain_S} is compact.  If $f$ is strictly
    positive on $\cS$, then $f \in \cT(\cS)$.
\end{theorem}

\begin{theorem}[Putinar Positivstellensatz] (\cite{Put:93})
    \label{thm: Putinar}
    Suppose $\cS$ as specified in \eqref{eq:domain_S} satisfies the following Archimedean
    condition: there exists $R>0$ such that
    \begin{equation*}
        R-\|\bx\|^2\in\cQ(\cS).
    \end{equation*}
    If $f$ is strictly positive on $\cS$, then
    $f\in\cQ(\cS)$.
\end{theorem}

In particular, the existence of Positivstellens\"atze effectively justifies the use of \eqref{eq:SOS} as hierarchies that solve polynomial optimization instances.

\subsection{Moment SOS hierarchy}

The moment-SOS hierarchy is a different hierarchy for polynomial optimization instances, and may be broadly understood as being dual to the SOS hierarchy. 

To describe these ideas, we let $\bv(\bx)$ denote the formal column vector
containing all standard monomials in $\bx$, and we let $\bv_r(\bx)$ denote the
column vector containing all monomials of degree at most $r$:
\begin{displaymath}
    \bv(\bx)
    =
    \big(\bx^\alpha\big)_{\alpha\in\NN^n}^\top,
    \qquad
    \bv_r(\bx)
    =
    \big(\bx^\alpha\big)_{\alpha\in\NN_r^n}^\top.
\end{displaymath}
Note that the basis vector $\bv_r(\bx)$ has dimension $s(n,r):=\binom{n+r}{n}$.  Using the monomial basis, any polynomial $f\in\RR[\bx]$ with degree at most $d$ can be expressed as
\begin{displaymath}
    f(\bx)
    =
    \sum_{\alpha\in\NN_d^n}f_\alpha\bx^\alpha
    =
    \langle\mathbf{f},\bv_d(\bx)\rangle,
    \qquad
    \mathbf{f}\in\RR^{s(n,d)}.
\end{displaymath}

Given a real sequence
$\biy=(y_\alpha)_{\alpha\in\NN^n}$ indexed by the monomials in
$\bv(\bx)$, we define the Riesz linear functional
$\ell_{\biy}:\RR[\bx]\to\RR$ by
\begin{displaymath}
    f(\bx)
    =
    \sum_{\alpha\in\NN^n}f_\alpha\bx^\alpha
    \quad\longmapsto\quad
    \ell_{\biy}(f)
    =
    \sum_{\alpha\in\NN^n}f_\alpha y_\alpha.
\end{displaymath}
The Riesz linear functional plays a central role in determining whether
a sequence $\biy$ is a moment sequence of a Borel measure; see, for instance, the Riesz--Haviland theorem in ~\cite[Theorem~3.1]{Las:09}.  We use the Riesz functional $\ell_{\biy}$ to define the moment and localizing matrices as follows.  Given an infinite sequence $\biy$ as in the above, the {\em moment matrix} $\bM(\biy)$, whose rows and columns are indexed by the monomials in
$\bv(\bx)$, is defined by
\begin{displaymath}
    \bM(\biy)(\alpha,\beta)
    =
    \ell_{\biy}\big(\bx^{\alpha+\beta}\big)
    =
    y_{\alpha+\beta},
    \qquad
    \forall\alpha,\beta\in\NN^n.
\end{displaymath}
Given $r\in\NN$, the $r$-{\em truncated moment matrix}, we denoted this by
$\bM_r(\biy)$, is the principal sub-matrix of $\bM(\biy)$ whose rows and columns are indexed by the monomials in $\bv_r(\bx)$.  Similarly, given a polynomial $g\in\RR[\bx]$, the localizing matrix $\bM(g\biy)$ associated with $\biy$ and $g$ is defined by
\begin{displaymath}
    \bM(g\biy)(\alpha,\beta)
    =
    \ell_{\biy}\big(g(\bx)\bx^{\alpha+\beta}\big)
    =
    \sum_{\gamma\in\NN^n}
    g_\gamma y_{\gamma+\alpha+\beta},
    \qquad
    \forall\alpha,\beta\in\NN^n.
\end{displaymath}
The $r$-{\em truncated localizing matrix} $\bM_r(g\biy)$ is principal sub-matrix of $\bM(g\biy)$ corresponding to the rows and columns indexed by monomials in $\bv_r(\bx)$.

Finally, given any $r\in\NN$ satisfying
\begin{displaymath}
    r
    \geq
    \max\big\{
        \lceil f\rceil,
        \lceil g_1\rceil,\ldots,\lceil g_m\rceil,
        \lceil h_1\rceil,\ldots,\lceil h_p\rceil
    \big\},
\end{displaymath}
we define the primal problem of the Schm\"udgen-type
moment--SOS hierarchy (based on the preordering) as follows:
\begin{align}
    f^{(r)}_{\rm mSOS,PO}
    :=
    \inf\Big\{
        \ell_{\biy}(f)
        =
        \sum_{\alpha\in\NN_{2r}^n}f_\alpha y_\alpha
        :
        \biy\in\pMTX2r
    \Big\}.
    \label{hierarchy: moment Schmudgen}
\end{align}
Here,
\begin{align}
    \pMTX2r
    :=
    \Big\{
        \biy\in\RR^{s(n,2r)}
        :\;&
        y_0=1,\quad
        \bM_r(\biy)\succeq0,
        \nonumber\\
        &
        \bM_{r-\lceil h_i\rceil}(h_i\biy)=0
        \quad
        \forall i\in[p],
        \nonumber\\
        &
        \bM_{r-\lceil g_J\rceil}(g_J\biy)\succeq0
        \quad
        \forall J\subseteq[m]
        \text{ such that }\lceil g_J\rceil\leq r
    \Big\}.
    \label{eq:pseudo-moment-set}
\end{align}
The primal problem of the Putinar-type moment--SOS hierarchy (based on the quadratic module) can be defined in an analogous fashion.  The elements of $\pMTX2r$ are known as {\em pseudo-moment sequences}, truncated at degree $2r$ (or also, of order $r$).  %For simplicity, we only discuss the Schm\"udgen-type hierarchy, which will be used throughout the remainder of this paper.

%[DO THESE DEFINITIONS COINCIDE WITH THE SOS HIERARCHY?]

\section{An SOS-hierarchy for quadratic optimization over measures} \label{sec: connection to POP}

In this section, we build on the preliminaries in Section \ref{sec:background} to provide intuition behind the construction of our hierarchy in Section~\ref{sec: contribution}.  

{\bf Approximating measures with atoms.}  Our first step is to gain intuition by approximating probability measures with atomic measures.  Consider the simpler instance where the marginal measures $\mu_i$ are discrete measures supported on $n_i$ atoms: 
\begin{equation*}
    \operatorname{supp}(\mu_i) = \{ \bx^{(i)}_1,\ldots,\bx^{(i)}_{n_i} \}.    
\end{equation*}
The marginal measure $\mu_i$ can then be represented by the probability vector
\begin{displaymath}
    \mu_i
    :=
    (\mu_{i1},\ldots,\mu_{in_i}),
    \qquad
    \mu_i
    =
    \sum_{j=1}^{n_i}
    \mu_{ij}\delta_{\bx^{(i)}_j}.
\end{displaymath}
Let
\begin{displaymath}
    \Im
    :=
    \big\{
        \gamma=(j_1,\ldots,j_m):
        j_i\in[n_i]\ \forall i\in[m]
    \big\},
    \qquad
    p:=|\Im|=\prod_{i=1}^m n_i.
\end{displaymath}
A feasible coupling can be represented by a finite-dimensional vector $\pi=(\pi_\gamma)_{\gamma\in\Im}\in\RR^p$.  The collection of all feasible coupling measures is then specified by the following polyhedral set:
\begin{multline*}
    \Pi(\mu_1,\ldots,\mu_m)
    =
    \Biggl\{
        \pi\in\RR^p:
        e_\gamma(\pi):=\pi_\gamma\geq0
        \quad\forall\gamma\in\Im,
        \\
        c_{k,l}(\pi)
        :=
        \sum_{\substack{
            \gamma=(j_1,\ldots,j_m)\in\Im\\
            j_k=l
        }}
        \pi_\gamma-\mu_{kl}
        =0
        \quad
        \forall k\in[m],\;
        l\in[n_k]
    \Biggr\}.
\end{multline*}

{\bf The moment--SOS hierarchy for atomic distributions.}  We proceed to specialize these ideas to our setting.  To this end, the prior work in \cite{tran:2026sum} studies a Schm\"udgen-type moment--SOS hierarchy for the discrete Gromov-Wasserstein problem, which is in effect our problem instance specialized to the case $m=2$.  By translating the ideas in \cite{tran:2026sum} to our set-up and subsequently extending these to general $m$, the Schm\"udgen-type moment--SOS hierarchy for our setting is as follows:
\begin{equation}
    \label{eq:schmudgen_dgw}
    \begin{aligned}
        \underset{\biy\in\RR^{s(p,2r)}}{\min}
        \qquad&
        \sum_{\gamma,\gamma'\in\Im}
        c_{\gamma,\gamma'}
        \ell_{\biy}(\pi_\gamma\pi_{\gamma'})
        \\
        \mbox{subject to}
        \qquad&
        y_0=1,
        \\
        &
        \overline{\bM}_{r-d_I}(e_I\biy)\succeq0
        \quad
        \forall I\subseteq\Im
        \text{ such that }d_I\leq r,
        \\
        &
        \ell_{\biy}
        \big(
            c_{k,l}(\pi)\pi^\alpha
        \big)
        =0
        \quad
        \forall\,|\alpha|\leq2r-1,\;
        k\in[m],\;
        l\in[n_k],
    \end{aligned}
    \tag{SDP-r}
\end{equation}
where
\begin{displaymath}
    e_I
    :=
    \prod_{\gamma\in I}e_\gamma,
    \qquad
    e_{\emptyset}:=1,
    \qquad
    d_I
    :=
    \left\lceil
        \frac{\deg(e_I)}{2}
    \right\rceil.
\end{displaymath}
Here, for every $I\subseteq\Im$ satisfying $d_I\leq r$,
$\overline{\bM}_{r-d_I}(e_I\biy)$ denotes the principal submatrix of
$\bM_{r-d_I}(e_I\biy)$ obtained by removing the rows and columns indexed by monomials $\pi^\alpha$ such that $|\alpha|<r-d_I$.  The work in \cite{tran:2026sum} investigates convergence rates associated with this hierarchy.

{\bf Reduction to a homogenous instance.}  The moment--SOS instance in \eqref{eq:schmudgen_dgw} involves all moments with degrees {\em up to} $2r$.  As it turns out, it is possible to reduce \eqref{eq:schmudgen_dgw} into a simpler instance which only involves moments with degree {\em equal to} $2r$ using the fact that $\pi$ is a probability distribution.

More precisely, let $S(\pi):=\sum_{\gamma\in\Im}\pi_\gamma$.  Because $\pi$ is a probability measure, it satisfies satisfies $S(\pi)=1$, which in turn implies $\ell_{\biy}(\pi^\alpha) = \ell_{\biy} (
\pi^\alpha\sum_{\gamma\in\Im}\pi_\gamma )$ for all $\alpha$.  By applying this recursively, we obtain the following for all $\alpha$ satisfying $|\alpha|<2r$:  
\begin{displaymath}
    \ell_{\biy}(\pi^\alpha)
=
    \ell_{\biy}\big(
        \pi^\alpha S(\pi)^{2r-|\alpha|}
    \big).
\end{displaymath}
In particular, this allows us to express moments with degree less than $2r$ as linear combinations of moments of degree exactly $2r$, resulting in the following reformulation of~\eqref{eq:schmudgen_dgw}:
\begin{equation}
\label{eq:schmudgen_dgw2}
\begin{aligned}
\underset{\biy\in\RR^{\overline{s}(p,2r)}}{\min}
\qquad&
    \sum_{\gamma,\gamma'\in\Im}
    c_{\gamma,\gamma'}
    \ell_{\biy}(\pi_\gamma\pi_{\gamma'})
\\
\mbox{subject to}
\qquad&
    \ell_{\biy}(1)=1,
\\
&
    \overline{\bM}_{r-d_I}(e_I\biy) \succeq 0
\quad \forall I\subseteq\Im \text{ such that } \deg(e_I) \leq r,
\\
&
    \ell_{\biy}\bigl(c_{k,l}(\pi)\pi^\alpha\bigr)=0
    \quad
    \forall\,|\alpha|\leq2r-1,\;
    k\in[m],\;
    l\in[n_k].
\end{aligned}
\tag{SDP-r2}
\end{equation}
Here, $\overline{\bv}_r(\pi)$ denotes the vector of all monomials in $\pi$ with degree exactly equal to $r$.

{\bf Translating moment--SOS to measure-centric constraints.}  From a geometric perspective, the constraints in~\eqref{eq:schmudgen_dgw}
can be viewed as one that provides an outer approximation of the set
\begin{displaymath}
    \operatorname{conv}
    \bigl\{
        \bv_{2r}(\pi):
        \pi\in\Pi(\mu_1,\ldots,\mu_m)
    \bigr\},
\end{displaymath}
while the constraints in~\eqref{eq:schmudgen_dgw2} can be viewed as one that provides an outer
approximation of the reduced moments
\begin{displaymath}
    \operatorname{conv}
    \bigl\{
        \overline{\bv}_{2r}(\pi):
        \pi\in\Pi(\mu_1,\ldots,\mu_m)
    \bigr\}.
\end{displaymath}
The latter description is convenient because the elements $\overline{\bv}_{2r}(\pi)$ have a natural interpretation as moments of the product measure
\begin{displaymath}
    \pi^{\otimes 2r}
    :=
    \underbrace{
        \pi\otimes\cdots\otimes\pi
    }_{2r\ \mathrm{times}}
    \in \cP(\cX^{2r}).
\end{displaymath}

Our next steps are to seek a formulation that is cast in the language of measures.  Following the prior discussion, the appropriate space to operate over is $\cX^{2r}$.  We let $\bP \in \cP(\cX^{2r})$ denote the decision variable.  Because the formulation~\eqref{eq:schmudgen_dgw2} has a natural interpretation as moments of the product measure $ \pi^{\otimes 2r}$, a natural approach is to introduce constraints to encourage $\bP = \pi^{\otimes 2r}$.  In fact, we should design constraints so that the eventual formulation coincides with~\eqref{eq:schmudgen_dgw2}.  By doing so, we may hope that the convergence properties commonly associated with the moment--SOS hierarchy readily translate over to our formulation over measures.  We explain this process in the following.

[Symmetry]:  Given a collection of atoms $\bx_1,\ldots,\bx_{2r}\in\cX$ and any permutation $\sigma\in S_{2r}$, we have
\begin{displaymath}
    \pi^{\otimes 2r}
    \bigl(
        \{\bx_1\}\times\cdots\times\{\bx_{2r}\}
    \bigr)
    =
    \prod_{i=1}^{2r}\pi(\{\bx_i\})
    =
    \prod_{i=1}^{2r}\pi(\{\bx_{\sigma(i)}\})
    =
    \pi^{\otimes 2r}
    \bigl(
        \{\bx_{\sigma(1)}\}\times\cdots
        \times\{\bx_{\sigma(2r)}\}
    \bigr).
\end{displaymath}
The measure $\bP$ should satisfy a similar property.  Note that the decision variable $\overline{\bv}_{2r}(\pi)$ accounts for this symmetry in the sense that the coordinates of $\overline{\bv}$ differ if and only if they are not equivalent via any permutation over the indices.  The natural extension of the above constraint for $\bP$ --  adapted for general Borel inputs beyond atomic masses -- is exactly \eqref{eq:cond_psd}:
\begin{equation}
    \label{eq:constraint_sym}
    \bP\bigl(C_1\times\cdots\times C_{2r}\bigr)
    =
    \bP\bigl(
        C_{\sigma(1)}\times\cdots\times C_{\sigma(2r)}
    \bigr)
    \qquad
    \forall\,\sigma\in S_{2r},\quad
    C_1,\ldots,C_{2r}\in\cB(\cX).
\end{equation}
%Equivalently, $\bP$ is invariant under every permutation of its $2r$ coordinates.

[Marginal condition]:  In the atomic setting, the sequence $\biy$ is indexed by monomials $\pi^\alpha$ with degree equals to $2r$. Each coordinate of $\biy$
\begin{displaymath}
    y_\alpha
    =
    \ell_{\biy}(\pi^\alpha),
    \qquad
    |\alpha|=2r,
\end{displaymath}
represents the probability mass assigned by $\bP$ to an associated atom of $\cX^{2r}$.  
The marginal
\begin{displaymath}
    \ell_{\biy}\bigl(c_{k,l}(\pi)\pi^\alpha\bigr)=0
    \qquad
    \forall\,|\alpha|\leq2r-1,\quad
    k\in[m],\quad l\in[n_k],
\end{displaymath}
admit the following measure-wise interpretation:
\begin{equation}
    \label{eq:constraint_mar}
    \bP\big|_{\cX_k\times\cX^{2r-1}}
    =
    \mu_k\otimes
    \bP\big|_{\cX^{2r-1}}
    \qquad
    \forall\,k\in[m].
\end{equation}
Here, the restriction notation denotes the corresponding marginal projection of $\bP$.  %In words, the $\cX_k$-component of one coordinate has distribution $\mu_k$ and is independent of the remaining $2r-1$ coordinates. %By the symmetry of $\bP$, the same condition holds for any of its $2r$ coordinates. %Hence,~\eqref{eq:constraint_mar} is precisely the atomic counterpart of the marginal condition~\eqref{eq:cond_mar}.

[Positive semidefinite conditions]:  In the atomic setting \eqref{eq:schmudgen_dgw2}, the positive semidefinite constraint associated with the moment matrix can be equivalently expressed as follows
\begin{equation*}
    \overline{\bM}_r(\biy)
    =
    \ell_{\biy}\bigl(
        \overline{\bv}_r(\pi)\overline{\bv}_r(\pi)^\top
    \bigr)
    \succeq 0
    \quad\Longleftrightarrow\quad
    \ell_{\biy}\bigl(f^2\bigr)\geq 0
    \quad
    \forall\,f\in\RR[\pi]_r.
\end{equation*}
This condition is induced by the fact that $\overline{\bv}_r(\pi)\overline{\bv}_r(\pi)^\top$ is PSD.  By identifying the homogeneous monomials of degree $r$ with the
coordinates of the product measure $\pi^{\otimes r}$, the appropriate extension of the condition $\ell_{\biy}(f^2)\geq 0$ to measures is
\begin{equation*}
    \int_{\cX^{2r}}
        f(\bx_1,\ldots,\bx_r)
        f(\bx_{r+1},\ldots,\bx_{2r})
        \,d\pi^{\otimes 2r}
    \;=\;
    \Big(
        \int_{\cX^r}
            f(\bx_1,\ldots,\bx_r)
            \,d\pi^{\otimes r}
    \Big)^2
    \geq 0.
\end{equation*}
Consequently, the natural extension of the moment-matrix constraint for probability measures $\bP\in\cP(\cX^{2r})$ is to impose
\begin{displaymath}
    \int_{\cX^{2r}}
        f(\bx_1,\ldots,\bx_r)
        f(\bx_{r+1},\ldots,\bx_{2r})
        \,d\bP
    \geq 0
    \qquad
    \forall\,f\in\mathfrak{B}_b(\cX^r).
\end{displaymath}
Here, $\mathfrak{B}_b(E)$ denotes the space of bounded real-valued
Borel functions over $E$.

In a similar fashion, the positive semidefinite constraints associated with the localizing matrices can be written as
\begin{equation*}
    \overline{\bM}_{r-d_I}(e_I\biy)
    =
    \ell_{\biy}\left(
        e_I(\pi)
        \overline{\bv}_{r-d_I}(\pi)
        \overline{\bv}_{r-d_I}(\pi)^\top
    \right)
    \succeq 0,\quad
    \forall\,I \subset \Im,\quad
    \deg(e_I)=2d_I\leq 2r,
\end{equation*}
or, equivalently,
\begin{displaymath}
    \ell_{\biy}\bigl(e_I(\pi)f(\pi)^2\bigr)\geq0
    \qquad
    \forall\,I,\quad
    \deg(e_I)=2d_I\leq2r,\quad
    f\in\RR[\pi]_{r-d_I}.
\end{displaymath}
Here, we note that $e_I(\pi) \overline{\bv}_{r-d_I}(\pi) \overline{\bv}_{r-d_I}(\pi)^\top$ is PSD.  The monomial $e_I(\pi)$ is non-negative because it represents the mass assigned by $\pi^{\otimes 2d_I}$ to a prescribed atom of $\cX^{2d_I}$.

To extend this interpretation beyond discrete spaces, we replace the mass assigned to an atom with the mass assigned to a Borel set. For every $0\leq t\leq r$, every $f\in\mathfrak{B}_b(\cX^{r-t})$, and every
$A\in\cB(\cX^{2t})$, we have
\begin{align*}
    &
    \int_{\cX^{2r}}
        f(\bx_1,\ldots,\bx_{r-t})
        f(\bx_{r-t+1},\ldots,\bx_{2r-2t})
        \chi_A(\bx_{2r-2t+1},\ldots,\bx_{2r})
        \,d\pi^{\otimes 2r}
    \\
    &\qquad
    =
    \left(
        \int_{\cX^{r-t}]}
            f(\bx_1,\ldots,\bx_{r-t})
            \,d\pi^{\otimes r-t}
    \right)^2
    \pi^{\otimes2t}(A)
    \geq0.
\end{align*}
Here, $\chi_A$ denotes the indicator function of $A$. More generally, we may replace $\chi_A$ with any nonnegative bounded Borel-measurable function $g$ on $\cX^{2t}$.  In this case, we ask that
\begin{align*}
    &
    \int_{\cX^{2r}}
        f(\bx_1,\ldots,\bx_{r-t})
        f(\bx_{r-t+1},\ldots,\bx_{2r-2t})
        g(\bx_{2r-2t+1},\ldots,\bx_{2r})
        \,d\pi^{\otimes 2r}
    \\
    &\qquad
    =
    \left(
        \int_{\cX^t}
            f(\bx_1,\ldots,\bx_{r-t})
            \,d\pi^{\otimes r-t}
    \right)^2
    \int_{\cX^{2t}}
        g(\bx_{2r-2t+1},\ldots,\bx_{2r})
        \,d\pi^{\otimes(2t)}
    \geq0.
\end{align*}

Subsequently, we impose the following family of positive
semidefinite constraints on $\bP\in\cP(\cX^{2r})$:
\begin{multline}
    \label{eq:constraint_psd}
    \mathsf{B}_{t,g}(f,f) = \int_{\cX^{2r}}
        f(\bx_1,\ldots,\bx_{r-t})
        f(\bx_{r-t+1},\ldots,\bx_{2r-2t})
        g(\bx_{2r-2t+1},\ldots,\bx_{2r})
        \,d\bP
    \geq0
    \\
    \forall\,0\leq t\leq r,\qquad
    f\in\mathfrak{B}_b(\cX^{r-t}),\qquad
    g\in\mathfrak{B}_b^+(\cX^{2t}),
\end{multline}
where
\begin{displaymath}
    \mathfrak{B}_b^+(E)
    :=
    \bigl\{
        g\in\mathfrak{B}_b(E):g\geq0
    \bigr\}.
\end{displaymath}
%When $t=0$ and $g\equiv1$, condition~\eqref{eq:constraint_psd} reduces to the measure-valued counterpart of the usual moment-matrix constraint:
%\begin{displaymath}
%    \int_{\cX^{2r}}
%        f(\bx_1,\ldots,\bx_r)
%        f(\bx_{r+1},\ldots,\bx_{2r})
%        \,d\bP
%    \geq0
%    \qquad
%    \forall\,f\in\mathfrak{B}_b(\cX^r).
%\end{displaymath}
%By contrast, the case $t=r$ simply expresses the non-negativity of the integral of every nonnegative bounded Borel-measurable function with respect to $\bP$.

[Conclusion]:  Let $\Pi_{r}$ be the intersection of the conditions \eqref{eq:constraint_sym},\eqref{eq:constraint_mar},
and~\eqref{eq:constraint_psd}.  Then by noting that the objective is linear integral with respect to $\bP$, the resulting reformulation of \eqref{eq:schmudgen_dgw2} over the measures is:
\begin{equation}
    \label{eq:ot-r-th}
    \rho^{(r)}
    :=
    \inf_{\bP\in\Pi_r}
    \left\{
        c(\bP)
        :=
        \int_{\cX^{2r}}
            c(\bx_1,\bx_2)
            \,d\bP(\bx_1,\ldots,\bx_{2r})
    \right\}.
\end{equation}

\section{Convergence of the hierarchy} 

The goal of this section is to establish convergence of the sequence $\{ \rho^{(r)} \}_{r=1}^{\infty}$ towards $\ropt$.  First, we begin with a basic observation:

\begin{proposition} \label{thm:simpleLB}
The sequence $\{ \rho^{(r)} \}_{r=1}^{\infty}$ is monotonically increasing.  In addition, $\rho^{(r)} \leq \rho^*$ for all $r \in \NN$.
\end{proposition}

% Note that Proposition \ref{thm:simpleLB} applies more generally to the sequence $\{\rho^{(r)}\}$ and not just $\{\rho^{(r)}\}$.

%[ADAPT PROOF TO THIS RESTATEMENT]

\begin{proof}
Let $\pi$ be a feasible solution to \eqref{eq:general_quadform}.  One checks that $\pi^{\otimes 2r}$ is feasible in \eqref{eq:general-r-th} (these checks are straightforward and omitted).  Consider the solutions $\pi$ and $\pi^{\otimes 2r}$ in \eqref{eq:general_quadform} and \eqref{eq:general-r-th} respectively.  Both of these evaluate to the same objective.  By noting that every feasible $\pi$ induces a $\pi^{\otimes 2r}$ that is feasible in \eqref{eq:general-r-th}, and by minimizing over $\pi$, we conclude that $\rho^{(r)} \leq \rho^*$.

Suppose $s \leq t$.  It remains to show that $\rho^{(s)} \leq \rho^{(t)}$.  The proof is similar.  Let $\bP^{(t)}$ be a feasible solution to \eqref{eq:general-r-th} with $r=t$.  Suppose $s<t$.  One checks that $\bP^{(s)}:=\bP^{(t)}\big|_{\cX^{2s}}$ obtained by marginalizing out $2(t-s)$ dimensions is also feasible in \eqref{eq:general-r-th} with $r=s$, and evaluates to the same objective.  By minimizing over $\bP^{(t)}$, we conclude that $\rho^{(s)} \leq \rho^{(t)}$.
%   This comes directly from the fact that for any $\pi \in \Pi$, the tensorized measure $\pi^{\otimes 2r}$ belongs to $\Pi_r$. In other words, when we project $\Pi_r$ into $\cX^2$, we obtain an outer approximation of $\Pi$. The proof is based on direct calculations to check all conditions in $\Pi_r$, which we omit here.
\end{proof}

\subsection{Main convergence result}

In the remainder of this section, we focus on the optimization model \eqref{eq:ot_quadform}; that is $\{ \rho^{(r)} \}$.  Our goal is to establish convergence.  

\begin{assume}[Regularity] \label{assume:reg}
    We assume that the cost function $c$ is continuous over $\cX^2$.  Furthermore, we assume that there is a metric $d_j$ over the spaces $\cX_j$, as well as a metric $d$ over the joint space $\cX$ that varies continuously with respect to the $d_j$'s in the sense that $d \rightarrow 0$ whenever $d_j \rightarrow 0$ for all $j \in [m]$.
\end{assume}

The role of the metric $d_j$ is to allow us to construct a discretization of the space $\cX_j$, and in particular, quantify the ``fineness'' of the resulting discretization.  In this paper, we assume that the sets $\cX_j$'s are Borel subsets of $\mathbb{R}^{n_j}$, and hence a natural choice is take the $d_j$'s to be the Euclidean metric.  We choose to write our results in terms of $d$ because it simplifies notation, and it clarifies the dependency on the metric space structure within our analysis.  

The simplest setting in which convergence holds is when $\cX$ is compact.

\begin{theorem}\label{thm:convergence_of_ctshierarchy_compact}
Suppose $\cX$ is compact, and suppose Assumption \ref{assume:reg} holds.  Then $\rho^{(r)} \rightarrow \ropt$.
\end{theorem}

With a bit more effort, it is possible to extend the conclusion of Theorem \ref{thm:convergence_of_ctshierarchy_compact} to settings where $\cX$ is not compact using assumptions about the tightness of probability measures.  More precisely, suppose we have:
\begin{assume}[Tightness]  \label{assume:tightness}
For every $\varepsilon>0$, there exists a compact set $K_{\varepsilon} \subset \cX^2$ such that 
\begin{equation}\label{eq: tightness}
    \int_{\cX^2 \backslash K_{\varepsilon}} |c(\bx_1,\bx_2)|~d\bP \leq \varepsilon \quad \forall \ \bP \in \Pi_r,\quad \forall \ r \in \NN.
\end{equation}
\end{assume}

We make some brief remarks concerning Assumption \ref{assume:tightness}.  It is, for instance, satisfied for the Gromov-Wasserstein problem whenever the distances $d_{\cY}$, $d_{\cZ}$ are chosen to be the squared Euclidean distances, the lost function $c$ is chosen to be the squared loss, and where the mm-spaces $(\cY,d_{\cY},\mu_{\cY})$ and $(\cZ,d_{\cZ},\mu_{\cZ})$ have finite fourth moments; i.e., 
\begin{equation*}
    \int_{\cY^2} d_{\cY}^4(\biy,\biy')~d\mu_{\cY} \otimes \mu_{\cY} < \infty, \quad \mbox{and} \quad  \int_{\cZ^2} d_{\cZ}^4(\bz,\bz')~d\mu_{\cZ} \otimes \mu_{\cZ} < \infty.
\end{equation*}
Furthermore, this condition is satisfied under Carleman's condition, which is introduced later in this paper.
Thus, for instance when we consider the 2-GW problem with squared distance, the assumptions are satisfied based on the following evaluation.
\begin{equation*}
\begin{aligned}
    \int_{(\cY \times \cZ)^2} (d_\cY^2(\biy,\biy')-d_{\cZ}^2(\bz,\bz'))^2~d\bP \leq & \int_{(\cY \times \cZ )^2}d^4_\cY(\biy,\biy')~d\bP+\int_{(\cY \times \cZ )^2}d^2_{\cY}(\bz,\bz')d\bP \\
    \leq & \int_{\cY^2} d^4_\cY(\biy,\biy')~d\mu_{\cY} \otimes \mu_{\cY} + \int_{\cZ^2} d^4_\cZ(\bz,\bz')~d\mu_{\cZ} \otimes \mu_{\cZ}.
\end{aligned}
\end{equation*}
The last expression is independent of $\bP$, which leads to uniform tightness assumption. 
This is our main convergence result regarding \eqref{eq:ot-r-th}. 

\begin{theorem} \label{thm:convergence_tightness}
Suppose Assumptions \eqref{assume:reg} and \eqref{assume:tightness} hold.  Then the sequence $\{\rho^{(r)}\}_{r \in \NN}$ increases monotonically towards $\ropt$.
\end{theorem}

{\em Remark.}  Note that Assumption \ref{assume:tightness} includes the case where $\cX$ is compact.  In particular, the proof of Theorem \ref{thm:convergence_tightness} implies Theorem \ref{thm:convergence_of_ctshierarchy_compact} as a by-product.

{\bf Summary of proof of Theorem \ref{thm:convergence_tightness}.}  We give a high level description of the proof of Theorem \ref{thm:convergence_tightness}, which can be divided into 3 main steps descibed as follows: 
\begin{itemize}
    \item{\textbf{Step 1.}} We construct a compact subset $\cY$ of $\cX$ in the sense that calculating~\eqref{eq:ot_quadform} with the objective taking integral over $\cY^2$ approximates~\eqref{eq:ot_quadform}, where the tightness is controlled by a parameter $\varepsilon >0$. We then discretize the space $\cY$ up to some precision, controlled by the parameter $\delta$.  The discretization induces an approximation of the optimization instance \eqref{eq:ot_quadform} as a finite dimensional quadratic program.  In a similar fashion, we show that the same discretization also induces a similar approximation of \eqref{eq:ot-r-th} as a hierarchy of finite-dimensional convex optimization problem. 

    \item{\textbf{Step 2.}} We show that the finite-dimensional convex optimization problems obtained by discretizing~\eqref{eq:ot-r-th} in Step~1 are SDPs. These SDPs can be naturally interpreted as a moment--SOS-type hierarchy applied to the finite-dimensional QP approximation of~\eqref{eq:ot_quadform}. Crucially, this interpretation allows us to establish that, as the relaxation order tends to infinity, the optimal values of the SDPs converge to the optimal value of the corresponding QP.

    \item{\textbf{Step 3.}} The third step shows that optimal solutions of the discretized QP and its associated SDP relaxations can be used to construct approximate optimal solutions of~\eqref{eq:ot_quadform} and~\eqref{eq:ot-r-th}, respectively. More precisely, the corresponding objective values are shown to be close to the optimal values $\rho^*$ and $\rho^{(r)}$, with approximation errors controlled by the parameters $\varepsilon$ and $\delta$. Finally, by letting $\delta\downarrow 0$ and $\varepsilon\downarrow 0$ in the appropriate order, and subsequently taking $r\to\infty$, we obtain the convergence asserted in Theorem~\ref{thm:convergence_tightness}.
\end{itemize}

The crux of our argument lies in the third step -- this is precisely where we rely on the specific structure of the constraints described in \eqref{eq:coupling_defn}.  In short, the essential ingredient that \eqref{eq:ot_quadform} possesses is that solutions (couplings) obtained by solving discretized instances of the OT problem can be seamlessly viewed as feasible solutions to its continuous problem.  For the more general class problems described in \eqref{eq:general_quadform}, there is no guarantee that a similar procedure succeeds.

In the remainder of this section, we prove Theorem \eqref{thm:convergence_tightness} according to the above outline.  %On a first reading of this paper, one may omit the rest of the section and proceed to Section [REF]. 

\subsection{Step 1: Discretize the space}

Our first step is to discretize the space $\cX$.  First, we fix $\varepsilon > 0$.  For every $j \in [m]$, let $\cY_j \subset \cX_j$, be compact subsets such that the set $\cY := \cY(\varepsilon) = \prod_{j=1}^m \cY_j$ satisfies the tightness condition in~\eqref{eq: tightness}. To see why $\cY$ can be chosen in this form, consider the following construction:  Let $K_{\varepsilon}$ be as specified in \eqref{eq: tightness}.  Define $\cY_j := \{ y_j \cup y_{j+m} : (y_1,y_2,\ldots) \in K_{\varepsilon} \}$.  $\cY_j$ is compact because $K_{\varepsilon}$ is.  But $K_{\varepsilon} \subseteq (\prod_{j=1}^m \cY_j) \times (\prod_{j=1}^m \cY_j)$, and hence $\cY$ satisfies our tightness requirement.
%We next generalize the idea of concentration-extension. 

Second, we fix $\delta > 0$.  Since the $\cY_i$'s are compact, we can partition each of these into a finite number of subsets such that each subset has diameter smaller than $\delta$; i.e., 
\begin{displaymath}
    \cY_j = \bigsqcup_{i=1}^{k_j} \cY_{j,i},\quad  \operatorname{diam}(\cY_{j,i}) < \delta \quad \forall j \in [m]. 
\end{displaymath}
For $j \in [m]$, we define $\cY_{j,k_j+1} := \cX_j \backslash \cY_j$ (the ``leftovers'').  The sets $\cX$ and $\cY$ admit the following partitions 
\begin{equation*}
\begin{aligned}
   \cX & := \textstyle \bigsqcup_{i_1 \in [k_1+1],\; \dots,\; i_m \in [k_m+1]}\cY_{1,i_1} \times \dots \times \cY_{m,i_m}, \\
   \cY & := \textstyle \bigsqcup_{i_1 \in [k_1],\; \dots,\; i_m \in [k_m]}\cY_{1,i_1} \times \dots \times \cY_{m,i_m}.
\end{aligned}
\end{equation*}
To simplify notation, we write $\cY_{1,i_1} \times \dots \times \cY_{m,i_m} = \cY_{\alpha}$, where $\alpha$ denotes the corresponding tuple of indices (respectively, $\cX$).  We let $\cI$ and $\cJ$ denote the collection of tuples of indices of $\cX$ and $\cY$ respectively.  This allows us to write
\begin{equation*}
    \cY = \bigsqcup_{\alpha \in \cI} \cY_{\alpha}, \quad \mbox{and} \quad \cX = \bigsqcup_{\alpha \in \cJ} \cY_{\alpha}. 
\end{equation*}

In our subsequent analysis, we operate on product spaces and hence need partitions of these.  Let $s \in \NN$, and let $\cX^{s} = \cX \times \ldots \times \cX$ ($s$ copies).  The previously introduced partitions naturally induce partitions on product spaces as follows
\begin{equation*}
    \cY^s := \textstyle \bigsqcup_{\alpha_1,\; \dots,\;\alpha_s \in \cI} \cY_{\alpha_1} \times \dots \times \cY_{\alpha_m} \quad \mbox{and} \quad \cX^s := \textstyle  \bigsqcup_{\alpha_1,\; \dots,\;\alpha_s \in \cJ} \cY_{\alpha_1} \times \dots \times \cY_{\alpha_m}.
\end{equation*}
To simply notation, we let 
\begin{displaymath}
    \cI(s) = \{\boldsymbol{\alpha}:=(\alpha_1,\dots,\alpha_s):\; \alpha_i \in \cI \; \forall i \in [s]\},\quad 
    \cJ(s) = \{\boldsymbol{\alpha}:=(\alpha_1,\dots,\alpha_s):\; \alpha_i \in \cJ \; \forall i \in [s]\}.
\end{displaymath}
Then $\cI(s)$ and $\cJ(s)$ are the indexation sets of the constructed partitions of $\cY^s$ and $\cX^s$, respectively.  Subsequently, we may write
\begin{equation*}
    \cY^s = \bigsqcup_{\boldsymbol \alpha \in \cI(s)} \cY_{\boldsymbol \alpha}, \quad \mbox{and} \quad \cX^s = \bigsqcup_{\boldsymbol \alpha \in \cJ(s)} \cY_{\boldsymbol \alpha},
\end{equation*}
where $\cY_{\boldsymbol\alpha}$ denote $\cY_{\alpha_1} \times \dots \times \cY_{\alpha_m}$.

As a consequence of Assumption \ref{assume:reg}, we have that $\{\cY_{\boldsymbol \alpha}\}_{\boldsymbol \alpha \in \cI(s)}$ forms a $d^{(s)}(\delta)$-cover of $\cY^s$, for some function $d^{(s)}(\delta)$ satisfying $\lim_{\delta \to 0} d^{(s)}(\delta) = 0$.

In the next part of the construction, we will access specific points in $\cX^s$.  To simplify notation, we adopt the following notation:  Given $\bx = (\bx_1,\dots,\bx_s) \in \cX^s$, we denote $\bx_{(i:j)} = (\bx_i,\dots,\bx_j)$.  Given $r \in \NN$ and $\balpha \in \cJ(s)$, we choose from $\cY_{\balpha}$ a representative point $\bx_{\balpha}$, and we define $\bP_{\balpha} = \bP(\cY_{\balpha})$.  
Note that these constructions are dependent of the parameter $\delta$.  To simplify notation, however, we omit the dependency on $\delta$.

These discretizations allow us to approximate the continuous objective with the discrete objective.  First, recall in Assumption \ref{assume:reg} that we have assumed $c$ to be continuous.  Second, recall that for every $\varepsilon>0$, one can choose $\cY := \cY(\varepsilon)$ such that it is compact.  Fix a sequence of values $\{\varepsilon_i\}_{i=1}^{\infty}$ where $\varepsilon_i \rightarrow 0$ decreases to zero.  For such a sequence, let $L(\varepsilon)$, $\varepsilon \in \{\varepsilon_i\}_{i=1}^{\infty}$, be the Lipschitz constant of $c$ over the choice of the set $\cY(\varepsilon)^2$.  Subsequently we then have the following bounds for any $\bP \in \Pi_r$, for all $r \in \NN$:

\begin{align}
    &\left|\int_{\cY^{2r}}c(\bx_{(1,2)})~d\bP- \sum_{\boldsymbol{\alpha} \in \cI(2r)}c(\bx_{\alpha,(1,2)})\bP_{\boldsymbol{\alpha}}\right| \leq L(\varepsilon)d^{(2)}(\delta),\label{eq:compact_approx_err}\\
    &\left|\int_{\cX^{2r}}c(\bx_{(1,2)})~d\bP- \sum_{\boldsymbol{\alpha} \in \cI(2r)}c(\bx_{\alpha,(1,2)})\bP_{\boldsymbol{\alpha}}\right| \leq L(\varepsilon)d^{(2)}(\delta) + \varepsilon\label{eq:discrete_approx_err}.
\end{align}

In particular, equation~\eqref{eq:compact_approx_err} follows from two key observations. First, the set $\cY^{2r}$ admits the partition
\begin{displaymath}
    \cY^{2r}
    =
    \bigsqcup_{\boldsymbol{\alpha}\in\cI(2r)}
    \cY_{\boldsymbol{\alpha}}.
\end{displaymath}
Second, by construction of this partition, the first two block coordinates of each cell $\cY_{\boldsymbol{\alpha}}$ are contained in a unique cell of the induced partition of $\cY^2$, whose diameter is at most $d^{(2)}(\delta)$. Recall that for any $\boldsymbol{\alpha} \in \cI(2r)$, $\bx_{\boldsymbol{\alpha},(1,2)}$ denotes the corresponding representative point. By Assumption~\ref{assume:reg}, the function $c$ is $L(\varepsilon)$--Lipschitz continuous on $\cY(\varepsilon)^2$. Consequently,
\begin{equation*}
    \begin{aligned}
       \left|\sum_{\boldsymbol{\alpha} \in \cI(2r)}\left( \int_{\cY_{\boldsymbol{\alpha}}}c(\bx_{(1,2)})~d\bP- c(\bx_{\alpha,(1,2)})\bP_{\boldsymbol{\alpha}}\right) \right| \leq L(\varepsilon)\delta\cdot \sum_{\boldsymbol{\alpha} \in \cI(2r)}\bP_{\boldsymbol{\alpha}} = L(\varepsilon)\delta\cdot \cP(\cY^{2r}) \leq L(\varepsilon)\delta,
    \end{aligned}
\end{equation*}
which proves~\eqref{eq:compact_approx_err}.

Combining~\eqref{eq:compact_approx_err} with Assumption~\ref{assume:tightness} yields
\begin{equation*}
    \begin{aligned}
        &\Biggl|\int_{\cX^{2r}}c(\bx_{(1,2)})~d\bP- \sum_{\boldsymbol{\alpha} \in \cI(2r)}c(\bx_{\alpha,(1,2)})\bP_{\boldsymbol{\alpha}}\Biggr|\\
        \leq \quad & \Biggl|\int_{\cY^{2r}}c(\bx_{(1,2)})~d\bP- \sum_{\boldsymbol{\alpha} \in \cI(2r)}c(\bx_{\alpha,(1,2)})\bP_{\boldsymbol{\alpha}}\Biggr| + \Biggl|\int_{\cX^{2r}}c(\bx_{(1,2)})~d\bP- \int_{\cY^{2r}}c(\bx_{(1,2)})~d\bP\Biggr|\\
        \leq \quad & L(\varepsilon)d^{(2)}(\delta) + \varepsilon.
    \end{aligned}
\end{equation*}
This establishes~\eqref{eq:discrete_approx_err}.

We next write down the finite dimensional quadratic program that approximates~\eqref{eq:ot_quadform}.  Recall the discretization on the marginal spaces $\cX_j = \bigsqcup_{i=1}^{k_j+1} \cY_{j,i} \; \forall j \in [m]$. To simplify notation, we denote $\overline{k_j} = k_j +1 \; \forall j \in [m]$.  Consider the following set
\begin{multline*}
\Delta(\varepsilon,\delta) := \Bigl\{\pi=(\pi_{i_1,\dots,i_m})_{i_1 \in \overline{k_1},\dots,i_m \in \overline{k_m}} :\pi_{i_1,\dots,i_m} \geq 0,\; \forall i_1 \in [\overline{k_1}],\dots, i_m \in [\overline{k_m}],\\ \sum_{l \neq j,\; i_l \in \overline{k_l}}\pi_{i_1,\dots,i_m} = \mu_{j,i_j}:= \mu(\cY_{j,i_j}) \; \forall j \in [m],\; i_j \in \overline{k_j} \Bigr\}.  
\end{multline*}
In words, $\Delta(\varepsilon,\delta)$ is the set of joint distributions (couplings) whose marginals equal to probability distributions arising from the discretization of the $\cX_j$'s.  In what follows, we write 
\begin{displaymath}
\pi = (\pi_\alpha)_{\alpha \in \cJ} \in \Delta(\varepsilon,\delta) \in \RR^{|\cJ|}.
\end{displaymath}
We define $\bc(\varepsilon,\delta) := (c(\bx_{(\alpha_1,\alpha_2)}))_{(\alpha_1,\alpha_2) \in \cI(2)}$, and we define
\begin{equation*}
\langle \bc(\varepsilon,\delta), \pi \otimes \pi\rangle = \sum_{\alpha_1,\alpha_2 \in \cI} c(\bx_{(\alpha_1,\alpha_2)})\pi_{\alpha_1} \pi_{\alpha_2}.
\end{equation*}
Consider the following optimization instance
\begin{equation}\label{eq:discrete_qp}
\rho(\varepsilon,\delta) := \min \; \langle \bc(\varepsilon,\delta), \pi \otimes \pi\rangle \quad \mbox{subject to}\quad \pi \in \Delta(\varepsilon,\delta).
\end{equation}
This is a finite dimensional quadratic program.  It is to be understood as the discrete analog of \eqref{eq:ot_quadform}.

We have showed that discretizations on $\cX$ incudes a quadratic program approximating~\eqref{eq:ot_quadform}. The remaining of this step is to show that the discretizations provides convex optimization problem approximating~\eqref{eq:ot-r-th}. For each positive integer $r$, and aforementioned discretization of $\cX_i$ for $i \in [m]$, we consider feasible solution $\bP \in \Pi_r$ of~\eqref{eq:ot-r-th}. Let $\boldsymbol{\beta}$ be a multi-index vector of $\NN^{|\cJ|}$ of the construction 
\begin{displaymath}
    \boldsymbol{\beta} := (\beta_{\alpha})_{\alpha \in \cJ}, \quad \beta_{\alpha} \in \NN\; \forall \; \alpha \in \cJ,
\end{displaymath}
and the length of a multi-index vector, denoted by $|\boldsymbol{\beta}|$, is defined by 
\begin{displaymath}
    |\boldsymbol{\beta}| := \sum_{\alpha \in \cJ} \beta_{\alpha}.
\end{displaymath}
For a positive integer $d$, we define 
\begin{displaymath}
    \NN^{|\cJ|}_{d} := \{ \boldsymbol{\beta} \in \NN^{|\cJ|}:\; |\boldsymbol{\beta}| \leq d\}, \quad \mbox{and} \quad
    \overline{\NN}^{|\cJ|}_{d} := \{ \boldsymbol{\beta} \in \NN^{|\cJ|}:\; |\boldsymbol{\beta}| = d\}.
\end{displaymath}
then $\bigl| \NN^{|\cJ|}_{d} \bigr|= s(|\cJ|,d)$, and $\bigl| \overline{\NN}^{|\cJ|}_{d} \bigr|= \overline{s}(|\cJ|,d)$ according to notation of multi-index vector in the moment-SOS hierarchy in Section~\ref{sec: connection to POP}. 

Consider $\bP \in \pi_r$. Recall that $\bP$ is symmetric, then for $\boldsymbol{\alpha}= (\alpha_1,\ldots,\alpha_{2r}) \in \cJ(2r)$, we have 
\begin{displaymath}
    \bP(\cY_{\boldsymbol{\alpha}}) = \bP(\cY_{\boldsymbol{\sigma(\alpha)}})\, \quad \mbox{where} \quad 
    \sigma(\boldsymbol{\alpha}) := (\alpha_{\sigma(1)},\ldots,\alpha_{\sigma(2r)}) \quad \forall \; \sigma \in S_{2r}.
\end{displaymath}
For such $\boldsymbol{\alpha} \in \cJ(2r)$, let define $\boldsymbol{\beta}(\boldsymbol{\alpha}) \in \overline{\NN}^{|\cJ|}_{2r}$ by 
\begin{equation}\label{def: alpha to beta}
    \boldsymbol{\beta}(\boldsymbol{\alpha}) := (\beta_\alpha)_{\alpha \in \cJ},\quad \text{where $\beta_{\alpha}$ is the number of appearances of $\alpha$ in $\boldsymbol{\alpha}$}.
\end{equation}

One can clearly observe that for any $\boldsymbol{\beta} \in \overline{\NN}^{|\cJ|}_{2r}$, there exists $\boldsymbol{\alpha} \in \cJ(2r)$ that $\boldsymbol{\beta}=\boldsymbol{\alpha}$. Take the following construction as an instance, for any $\boldsymbol{\beta} \in \overline{\NN}^{|\cJ|}$, we define 
\begin{equation}\label{def: beta to alpha}
    \boldsymbol{\alpha}(\boldsymbol{\beta}) = (\ldots,\underbrace{\alpha,\ldots,\alpha}_{\beta_{\alpha} \text{ times}},\ldots) \in \cJ(|\boldsymbol{\beta}|),\quad
    \mbox{then} \quad 
    \boldsymbol{\beta}(\boldsymbol{\alpha}) = \boldsymbol{\beta}.
\end{equation}
Hence, there is an one-to-one correspondence between $\overline{\NN}^{|\cJ|}_{2r}$ and $\cJ(2r) / \sim $, where $\sim$ denotes an equivalent relationship induced by the permutation set $S_{2r}$. Since $\bP$ is symmetric, we define 
\begin{displaymath}
    \biy(\bP) := (\biy_{\boldsymbol{\beta}})_{\boldsymbol{\beta} \in \overline{\NN}^{|\cJ|}_{2r}},\quad \mbox{where} \quad \biy_{\boldsymbol{\beta}} = \bP_{\boldsymbol{\alpha}}\; \mbox{ such that }\; \boldsymbol{\beta}(\boldsymbol{\alpha})=\boldsymbol{\beta}.
\end{displaymath}

We recall that in~\eqref{eq:discrete_approx_err}, the objective of~\eqref{eq:ot-r-th} is approximated by 
\begin{displaymath}
    \sum_{\boldsymbol{\alpha} \in \cI(2r)}c(\bx_{\alpha,(1,2)})\bP_{\boldsymbol{\alpha}},
\end{displaymath}
which can be reformulated in terms of $\biy$ in the fashion: let us define 
\begin{displaymath}
    \langle \bc(\varepsilon,\delta), \biy(\bP) \rangle = \sum_{\boldsymbol{\alpha} \in \cI(2r)}c(\bx_{\alpha,(1,2)})\iy_{\boldsymbol{\beta}(\boldsymbol{\alpha})}= \sum_{\boldsymbol{\alpha} \in \cI(2r)}c(\bx_{\alpha,(1,2)})\bP_{\boldsymbol{\alpha}}.
\end{displaymath}
Then the problem~\eqref{eq:ot-r-th} admits an approximation by discretization, defined as
\begin{equation}\label{eq:discrete_convexp}
\rho^{(r)}(\varepsilon,\delta) := \min \; \langle \bc(\varepsilon,\delta), \biy \rangle \quad \mbox{subject to}\quad \biy \in \biy(\Pi_r):= \{\biy(\bP),\; \bP \in \Pi_r\}.
\end{equation}
This problem is a finite-dimensional convex optimization problem since the objective is linear in $\biy$, the set $\biy(\Pi_r)$ can be interpreted as a "projection" of $\Pi_r$ by taking mass of partitions sets, which inherits the convexity of $\Pi_r$.

In summary, this step demonstrates that the discretization induces the quadratic problem~\eqref{eq:discrete_qp} and the convex optimization~\eqref{eq:discrete_convexp} to approximate~\eqref{eq:ot_quadform} and~\eqref{eq:ot-r-th} respectively.

\subsection{Step 2: Verifying SDP relaxation of \texorpdfstring{~\eqref{eq:discrete_convexp}}{(1)}}

The goal of this step is to show that~\eqref{eq:discrete_convexp} may be viewed as the semidefinite relaxation obtained by applying the moment--SOS hierarchy to the quadratic programming instance~\eqref{eq:discrete_qp}.  Following our discussion in Section~\ref{sec: connection to POP}, the quadratic programming instance~\eqref{eq:discrete_qp} is a polynomial optimization instance in which the objective is quadratic and where the constraint set is the following polyhedral set 
\begin{displaymath}
    \pi \in \Delta(\varepsilon,\delta) = \Pi(\overline{\mu}_i, \dots,\overline{\mu}_m),
\end{displaymath}
where $\overline{\mu}_j$ is an $\overline{k}_j$--dimensional distributional vector from the discretization in Step~1.  Given a positive integer $r$, the $r$--level of the moment--SOS hierarchy of~\eqref{eq:discrete_qp} is given by 
\begin{equation}
 \label{eq:schmudgen_dgw3}
\begin{aligned}
\underset{\biy\in\RR^{\overline{s}(|\cJ|,2r)}}{\min}
\qquad&
    \sum_{\alpha_1,\alpha_2\in\cI}c(\bx_{\alpha_1},\bx_{\alpha_2})
    \ell_{\biy}(\pi_{\alpha_1}\pi_{\alpha_2}) = \langle \bc(\varepsilon,\delta), \biy \rangle 
\\
\mbox{subject to}
\qquad&
\ell_{\biy}(1)=1,
\\
&
    \overline{\bM}_{r-t}(e_I(\pi)\biy)
    :=
    \ell_{\biy}\left(
        e_I(\pi)\,
        \overline{\bv}_{r-t}(\pi)
        \overline{\bv}_{r-t}(\pi)^\top
    \right)
    \succeq0 \quad \forall\, I \subset \cJ,\; |I| = 2t \leq 2r.
\\
&
    \ell_{\biy}\bigl(c_{j,i}(\pi)\pi^{\boldsymbol{\beta}}\bigr)=0
    \quad
    \forall\,|\boldsymbol{\beta}|\leq2r-1,\;
    j\in[m],\;
    i\in[\overline{k}_i].
\end{aligned}
\tag{SDP-r3}
\end{equation}
For any $I \subset J$, we may define $\boldsymbol{\beta}(I) = (\chi_{I}(\alpha))_{\alpha \in \cJ}$ so that $e_I(\pi) = \pi^{\boldsymbol{\beta}(I)}$.  The SDP relaxation \eqref{eq:schmudgen_dgw3} is similar to~\eqref{eq:schmudgen_dgw2} in which the Riesz linear functional associated to $\biy$ is defined by 
\begin{displaymath}
  \ell_{\biy}(\pi^{\boldsymbol{\beta}})
    :=
    y_{\boldsymbol{\beta}}
    \quad
    \forall\,|\boldsymbol{\beta}|=2r, \qquad 
    \ell_{\biy}(\pi^{\boldsymbol{\beta}})
    :=
    \ell_{\biy}\left(
        \pi^{\boldsymbol{\beta}} S(\pi)^{2r-|\boldsymbol{\beta}|}
    \right) \quad \forall \,|\boldsymbol{\beta}|<2r, \qquad S(\pi) = \sum_{\alpha \in \cJ}\pi_{\alpha}.
\end{displaymath}
In particular, we formalize the connection between~\eqref{eq:discrete_convexp} and~\eqref{eq:schmudgen_dgw3} in the following.

\begin{theorem}\label{thm: connnection between relaxations}
    The feasible set of~\eqref{eq:schmudgen_dgw3} is $\biy(\Pi_r)$, which is the feasible set of~\eqref{eq:discrete_convexp}. Consequently, ~\eqref{eq:discrete_convexp} and~\eqref{eq:schmudgen_dgw3} coincide.
\end{theorem}

\begin{proof}
 For clarity, we recall the definitions of the two feasible sets under
consideration.

\begin{multicols}{2}

\noindent
\textbf{Feasible set of~\eqref{eq:schmudgen_dgw3}.}
\par\noindent\rule{\linewidth}{0.4pt}

\begin{align*}
& 
\biy\in\RR^{\overline{s}(|\cJ|,2r)},\quad \ell_{\biy}(1)=1,
\\
&
\overline{\bM}_{r-t}(e_I(\pi)\biy):=\ell_{\biy}\left(
e_I(\pi)\,
\overline{\bv}_{r-t}(\pi)
\overline{\bv}_{r-t}(\pi)^\top
\right)\\
&
\hspace{7em}\succeq0\quad
\forall\, I \subset \cJ,\; |I| = 2t \leq 2r,
\\
&
\ell_{\biy}
\bigl(c_{j,i}(\pi)\pi^{\boldsymbol{\beta}}\bigr)=0
\\[-1mm]
&
\quad
\forall\,|\boldsymbol{\beta}|\leq2r-1,\quad
j\in[m],\quad j\in[\overline{k}_i].
\end{align*}
\columnbreak

\noindent
\textbf{Feasible set of~\eqref{eq:discrete_convexp}.}
\par\noindent\rule{\linewidth}{0.4pt}

\begin{align*}
    \biy(\Pi_r)
    &:=
    \bigl\{
        \biy(\bP):
        \bP\in\Pi_r
    \bigr\},\\
    \Pi_r
    &:=
    \Bigl\{
        \bP\in\cP(\cX^{2r}):
        \bP \text{ satisfies }
        \eqref{eq:constraint_sym},\\
    & \hspace{9em}    
        \eqref{eq:constraint_mar},
        \text{ and }
        \eqref{eq:constraint_psd}
    \Bigr\}.
\end{align*}
\end{multicols}

The proof proceeds in two parts, each on which establishes a set inclusion.

\textbf{First part.} We prove that every $\biy \in \biy(\Pi_r)$ is a feasible solution of~\eqref{eq:schmudgen_dgw3}. Indeed, since $\biy \in \RR^{\overline{s}(|\cJ|,2r)}$, the Riesz linear functional associated $\biy$ is well-defined, and we need to prove that $\ell_{\biy}$ satisfies all constraints in~\eqref{eq:schmudgen_dgw3}. 

To prove $\ell_{\biy}(1) =1$, we consider
\begin{displaymath}
    \ell_{\biy}(1) = \ell_{\biy}(S(\pi)^{2r}) = \ell_{\biy}\Bigl(\sum_{\boldsymbol{\alpha} \in \cJ(2r)} \pi^{\boldsymbol{\beta}(\boldsymbol{\alpha})}\Bigr) = \sum_{\boldsymbol{\alpha} \in \cJ(2r)}\ell_{\biy}(\pi^{\boldsymbol{\beta}(\boldsymbol{\alpha})}) = \sum_{\boldsymbol{\alpha} \in \cJ(2r)}\bP(\cY_{\boldsymbol{\alpha}})= \bP(\cX^{2r}) =1.
\end{displaymath}
Next, for any $I \in \cJ$ and $|I| =2t\leq 2r$, the positive semi-definiteness of $\overline{\bM}_{r-t}(e_I(\pi)\biy)$ is equivalent to 
\begin{displaymath}
    \mathbf{p}^\top \overline{\bM}_{r-t}(e_I(\pi)\biy) \mathbf{p} \,\geq\, 0 \quad \forall\, \mathbf{p} \in \RR^{\overline{s}(|\cJ|,r-t)}
\end{displaymath}
We fix $\mathbf{p} = (p_{\boldsymbol{\beta}})_{\boldsymbol{\beta} \in \NN^{|\cJ|}_{r-t}} \in \RR^{\overline{s}(|\cJ|,r-t)}$ and recall~\eqref{def: beta to alpha}, we define the following piece-wise functions 
\begin{equation*}
    \begin{aligned}
    & f(\bx_1,\dots,\bx_{r-t})= \sum_{\boldsymbol{\beta} \in \NN^{|\cJ|}_{r-t}}p_{\boldsymbol{\beta}} \cdot \chi_{\cY_{\boldsymbol{\alpha}(\boldsymbol{\beta})}}(\bx_{(1,r-t)}) \in \mathfrak{B}_b(\cX^{r-t}), \\
    & g(\bx_1,\dots,\bx_{2t}) = \chi_{\cY_{\boldsymbol{\alpha}(\boldsymbol{\beta(I)})}}(\bx_{(1,2t)}) \in \mathfrak{B}_b^+(\cX^{2t}).
\end{aligned}
\end{equation*}
According to the definition of $\overline{\bM}_{r-t}(e_I(\pi)\biy)$, we obtain 
\begin{align*}
    & \mathbf{p}^\top \overline{\bM}_{r-t}(e_I(\pi)\biy) \mathbf{p}\;=\; \ell_{\biy}\left(\mathbf{p}^\top \pi^{\boldsymbol{\beta}(I)}\,
    \overline{\bv}_{r-t}(\pi)
    \overline{\bv}_{r-t}(\pi)^\top \mathbf{p} \right)\\
    =\;& \ell_{\biy}\left(\sum_{\boldsymbol{\beta}, \boldsymbol{\beta}^\prime \in \NN^{|\cJ|}_{r-t}}p_{\boldsymbol{\beta}}p_{\boldsymbol{\beta}^\prime}\pi^{\boldsymbol{\beta}(I)+\boldsymbol{\beta}+\boldsymbol{\beta}^\prime} \right)\;=\; \sum_{\boldsymbol{\alpha}, \overline{\boldsymbol{\beta}} \in \cJ(r-t)}p_{\boldsymbol{\alpha}}p_{\boldsymbol{\alpha}^\prime}\ell_{\biy}\big(\pi^{\boldsymbol{\beta}(I)+\boldsymbol{\beta}+\boldsymbol{\beta}^\prime}\bigr).
\end{align*}
We recall that since $\biy \in \biy(\Pi_r)$, and $|\boldsymbol{\alpha}(\boldsymbol{\beta}(I))|+|\boldsymbol{\alpha}(\boldsymbol{\beta})|+|\boldsymbol{\alpha}(\boldsymbol{\beta}^\prime)|=2t+r-t+r-t =2r$, then $(\boldsymbol{\alpha}(\boldsymbol{\beta}),\boldsymbol{\alpha}(\boldsymbol{\beta}^\prime),\boldsymbol{\alpha}(\boldsymbol{\beta}(I))) \in \cJ(2r)$, and
\begin{align*}
   \ell_{\biy}\bigl(\pi^{\boldsymbol{\beta}(I)+\boldsymbol{\beta}+\boldsymbol{\beta}^\prime}\bigr) \;=\;& \bP(\cY_{(\boldsymbol{\alpha}(\boldsymbol{\beta}),\boldsymbol{\alpha}(\boldsymbol{\beta}^\prime),\boldsymbol{\alpha}(\boldsymbol{\beta}(I)),)})\\
   =\; &\int_{\cX^{2r}}\chi_{\cY_{\boldsymbol{\alpha}(\boldsymbol{\beta})}}(\bx_{(1,r-t)})\cdot \chi_{\cY_{\boldsymbol{\alpha}(\boldsymbol{\beta}^\prime)}}(\bx_{(r-t+1,2r-2t)})\cdot\chi_{\cY_{\boldsymbol{\alpha}(\boldsymbol{\beta(I)})}}(\bx_{2t-2t+1,2r})~d\bP.
\end{align*}
Hence, we have 
\begin{displaymath}
    \mathbf{p}^\top \overline{\bM}_{r-t}(e_I(\pi)\biy) \mathbf{p}\;=\; \int_{\cX^{2r}}f(\bx_{(1,r-t)})f(\bx_{(r-t+1,2r-2t)})g(\bx_{2r-2t+1,2r)})~d\bP \geq 0
\end{displaymath}
according to positive semi-definiteness of $\bP$.

The remaining properties of $\ell_{\biy}$ we need to prove is 
\begin{displaymath}
    \ell_{\biy}\bigl(c_{j,i}(\pi)\pi^{\boldsymbol{\beta}}\bigr)=0 \quad
\forall\,|\boldsymbol{\beta}|\leq2r-1,\quad
j\in[m],\quad j\in[\overline{k}_i],
\end{displaymath}
which can be derived directly from the marginal condition of $\bP$, namely,
\begin{displaymath}
     \bP\big|_{\cX_j\times\cX^{2r-1}}
    =
    \mu_k\otimes
    \bP\big|_{\cX^{2r-1}}
    \qquad
    \forall\,j\in[m].
\end{displaymath}
This completes the first part of the proof.

\textbf{Second part.} We prove that any for any feasible solution $\biy$ of~\eqref{eq:schmudgen_dgw3}, there exists a $\bP \in \Pi_r$ such that $\biy = \biy(\bP)$. Indeed, we define the measure 
\begin{equation*}
    \bP(\biy) = \sum_{\boldsymbol{\alpha} \in \cJ(2r)}\frac{\ell_{\biy}(\pi^{\boldsymbol{\beta}(\boldsymbol{\alpha})})}{\mu^{\otimes 2r}(\cY_{\boldsymbol{\alpha}})}\cdot \mu^{\otimes 2r}\big|_{\cY_{\boldsymbol{\alpha}}},
\end{equation*}
where
\begin{equation*}
    \mu^{\otimes 2r} = (\mu_1 \otimes \dots \otimes \mu_m) \otimes \dots \otimes (\mu_1 \otimes \dots \otimes \mu_m) \in \cP(\cX^{2r}).
\end{equation*}
The set $\Pi_r$ is specified as the intersection of the conditions \eqref{eq:constraint_sym}, \eqref{eq:constraint_mar}, and \eqref{eq:constraint_psd}.  We check that $\bP(\biy)$ each of these.

[Symmetry]:  This follows from the fact that both $\biy$ and $\mu^{\otimes 2r}$ are symmetric, namely, 
\begin{displaymath}
    \mu^{\otimes 2r}(C) = \mu^{\otimes 2r}(\sigma(C)), \quad \boldsymbol{\beta}(\sigma(\boldsymbol{\alpha})) = \boldsymbol{\beta}(\boldsymbol{\alpha}), \quad \forall\, \sigma \in S_{2r},\, C \in \mathcal{B}(\cX^{2r}).
\end{displaymath}

[Marginal condition]:  This entails checking that the following holds for all $j \in [m]$:
\begin{equation*}
    \bP(\biy)\big|_{\cX_j \times \cX^{2r-1}} = \mu_j \otimes \bP(\biy)\big|_{\cX^{2r-1}}. 
\end{equation*}
To do so, we make the following evaluations 
\begin{displaymath}
    \bP(\biy)\big|_{\cX_j \times \cX^{2r-1}} = \sum_{\boldsymbol{\alpha} \in \cJ(2r)}\frac{\ell_{\biy}(\pi^{\boldsymbol{\beta}(\boldsymbol{\alpha})})}{\mu^{\otimes 2r}(\cY_{\boldsymbol{\alpha}})}\cdot (\mu^{\otimes 2r}\big|_{\cY_{\boldsymbol{\alpha}}})\big|_{\cX_j \times \cX^{2r-1}}
\end{displaymath}
Recall that for any $\boldsymbol{\alpha} \in \cJ(2r)$, one can write $\boldsymbol{\alpha} = (\alpha_1,\dots,\alpha_{2r})$ with $\alpha_j \in \cJ \; \forall\, j \in [m]$, and recall the projections $\mathrm{pr}_j:\; \cX \to \cX_j$ to have 
\begin{align*}
    (\mu^{\otimes 2r}\big|_{\cY_{\boldsymbol{\alpha}}})\big|_{\cX_j \times \cX^{2r-1}} =&
    \Bigl(\prod_{j \in [m],\; i \neq j}\mu_i(\mathrm{pr}_i(\cY_{\alpha_{1}}))\Bigr)\cdot \mu_j\big|_{\mathrm{pr}_j(\cY_{\alpha_{1}})}\otimes \mu^{2r-1}\big|_{\cY_{\alpha_{(2.2r)}}} \\
    \mu^{\otimes 2r}(\cY_{\boldsymbol{\alpha}}) =&  \left(\prod_{j=1}^m \mu_j(\mathrm{pr}_j(\cY_{\alpha_1}))\right) \cdot \mu^{2r-1}(\cY_{\boldsymbol{\alpha}}).
\end{align*}
We then have 
\begin{align*}
    \bP(\biy)\big|_{\cX_j \times \cX^{2r-1}} =& \sum_{\boldsymbol{\alpha} \in \cJ(2r)}\frac{\ell_{\biy}(\pi^{\boldsymbol{\beta}(\boldsymbol{\alpha})})}{\mu_j(\mathrm{pr}_j(\cY_{\alpha_{1}}))\cdot\mu^{2r-1}(\cY_{\boldsymbol{\alpha}})}\cdot\mu_j\big|_{\mathrm{pr}_j(\cY_{\alpha_{1}})}\otimes \mu^{2r-1}\big|_{\cY_{\alpha_{(2.2r)}}}\\
    =& \sum_{\boldsymbol{\alpha} \in \cJ(2r-1)}\sum_{\alpha_{1} \in \cJ}
    \frac{\ell_{\biy} (\pi^{\boldsymbol{\beta}(\alpha_{1})+\boldsymbol{\beta}(\boldsymbol{\alpha})})}{\mu_j(\mathrm{pr}_j(\cY_{\alpha_{1}}))\cdot\mu^{2r-1}(\cY_{\boldsymbol{\alpha}})}
    \cdot \mu_j\big|_{\mathrm{pr}_j(\cY_{\alpha_{1}})}\otimes \mu^{2r-1}\big|_{\cY_{\boldsymbol{\alpha}}},
\end{align*}
and for fixed $i \in [\overline{k}_j]$, we take all the $\alpha_1 \in \cJ$ such that $\mathrm{pr}_j(\cY_{\alpha_1}) = \cY_{i}$, then we obtain 
\begin{displaymath}
    \sum_{\substack{\alpha_1 \in \cJ\\ \mathrm{pr}_j(\cY_{\alpha_1}) = \cY_{i}}} \ell_{\biy} (\pi^{\boldsymbol{\beta}(\alpha_{1})+\boldsymbol{\beta}(\boldsymbol{\alpha})})
    =  \ell_{\biy} \left(\left(\sum_{\substack{\alpha_1 \in \cJ\\ \mathrm{pr}_j(\cY_{\alpha_1}) = \cY_{i}}}\pi^{\boldsymbol{\beta}(\alpha_1)} \right)\pi^{\boldsymbol{\beta}(\boldsymbol{\alpha})}\right)
    = \mu_j(\cY_{j,i}) \cdot \ell_{\biy}(\pi^{\boldsymbol{\beta}(\boldsymbol{\alpha})}).
\end{displaymath}
Hence, 
\begin{align*}
    \bP(\biy)\big|_{\cX_j \times \cX^{2r-1}} =&
    \sum_{\boldsymbol{\alpha} \in \cJ(2r-1)}\sum_{i \in [\overline{k}_j]}
    \frac{\mu_j(\cY_{j,i}) \cdot \ell_{\biy}(\pi^{\boldsymbol{\beta}(\boldsymbol{\alpha})})}{\mu_j(\cY_{j,i})\cdot\mu^{2r-1}(\cY_{\boldsymbol{\alpha}})}
    \cdot \mu_j\big|_{\cY_{j,i}}\otimes \mu^{2r-1}\big|_{\cY_{\boldsymbol{\alpha}}}\\
    =&\sum_{\boldsymbol{\alpha} \in \cJ(2r-1)}\sum_{i \in [\overline{k}_j]}
    \frac{\ell_{\biy}(\pi^{\boldsymbol{\beta}(\boldsymbol{\alpha})})}{\mu^{2r-1}(\cY_{\boldsymbol{\alpha}})}
    \cdot \mu_j\big|_{\cY_{j,i}}\otimes \mu^{2r-1}\big|_{\cY_{\boldsymbol{\alpha}}}\\
    =&\left(\sum_{i \in [\overline{k}_j]} \mu_j\big|_{\cY_{j,i}} \right) \otimes \left(\sum_{\boldsymbol{\alpha} \in \cJ(2r-1)} \frac{\ell_{\biy}(\pi^{\boldsymbol{\beta}(\boldsymbol{\alpha})})}{\mu^{2r-1}(\cY_{\boldsymbol{\alpha}})}\cdot\mu^{2r-1}\big|_{\cY_{\boldsymbol{\alpha}}} \right)\\
    =& \mu_j \otimes \left(\sum_{\boldsymbol{\alpha} \in \cJ(2r-1)} \frac{\ell_{\biy}(\pi^{\boldsymbol{\beta}(\boldsymbol{\alpha})})}{\mu^{2r-1}(\cY_{\boldsymbol{\alpha}})}\cdot\mu^{2r-1}\big|_{\cY_{\boldsymbol{\alpha}}} \right).
\end{align*}
We also note that from the definition of Riesz linear functional, the following holds
\begin{equation*}
     \sum_{ \alpha_1 \in \cJ}\ell_{\biy}(\pi^{\boldsymbol{\beta}(\alpha_1)+\boldsymbol{\beta}(\boldsymbol{\alpha})}) = \ell_{\biy}(S(\pi)\cdot\pi^{\boldsymbol{\beta}(\boldsymbol{\alpha})}) = \ell_{\biy}(\pi^{\boldsymbol{\alpha}}),
\end{equation*}
which implies that
\begin{align*}
    \bP(\biy)\big|_{\cX^{2r-1}} 
    &= \sum_{\boldsymbol{\alpha} \in \cJ(2r-1)}
    \sum_{\alpha_1 \in \cJ}\frac{\ell_{\biy}(\pi^{\boldsymbol{\beta}(\boldsymbol{\alpha})+\boldsymbol{\beta}(\alpha_1)})}
    {\mu(\cY_{\alpha_1})\cdot\mu^{\otimes 2r-1}(\cY_{\boldsymbol{\alpha}})}\cdot \mu_j(\cY_{\alpha_1})\cdot(\mu^{\otimes 2r-1}\big|_{\cY_{\boldsymbol{\alpha}}})\big|_{\cX^{2r-1}}\\
    &= \sum_{\boldsymbol{\alpha} \in \cJ(2r-1)} \frac{\ell_{\biy}(\pi^{\boldsymbol{\beta}(\boldsymbol{\alpha})})}{\mu^{2r-1}(\cY_{\boldsymbol{\alpha}})}\cdot\mu^{2r-1}\big|_{\cY_{\boldsymbol{\alpha}}}.
\end{align*}
This verifies the marginal constraint~\eqref{eq:constraint_mar}.

[Checking positive semidefinite-ness]: Checking the PSD condition entails showing that the following integral is non-negative for all $0 \leq t \leq r$, all $g \in \mathfrak{B}_b^+(\cX^{2t})$ and $f \in \mathfrak{B}_b(\cX^{r-t})$
\begin{equation*}
\mathrm{B}_{t,g}(f,f)=\int_{\cX^{2r}}f(\bx_{(1:r-t)})f(\bx_{(r-t+1:2r-2t)})g(\bx_{(2r-2t+1,2r)})~d\bP \geq 0.    
\end{equation*}
Following the definition of $\bP$, we can evaluate the integral $\mathrm{B}_{t,g}(f,f)$ as a sum as follows: 
\begin{align*}
    \quad \mathrm{B}_{t,g}(f,f)=&\quad 
    \sum_{\boldsymbol{\alpha} \in \cJ(2r)}
    \int_{\cY_{\boldsymbol{\alpha}}}
    f(\bx_{(1:r-t)})f(\bx_{(r-t+1:2r-2t)})g(\bx_{(2r-2t+1,2r)})~d\bP\big|_{\cY_{\boldsymbol{\alpha}}} \\
    =&\quad 
    \sum_{\substack{\boldsymbol{\alpha},\boldsymbol{\alpha}^\prime \in \cJ(r-t) \\ \boldsymbol{\alpha}^{\prime \prime} \in \cJ(2t)}}
    \frac{\ell_{\biy}(\pi^{\boldsymbol{\beta}(\boldsymbol{\alpha},\boldsymbol{\alpha}^\prime,\boldsymbol{\alpha}^{\prime \prime})})}
    {\mu^{\otimes r-t}(\cY_{\boldsymbol{\alpha}}) \cdot \mu^{\otimes r-t}(\cY_{\boldsymbol{\alpha}^\prime})\cdot \mu^{\otimes 2t}(\cY_{\boldsymbol{\alpha}^{\prime \prime}})}\cdot\\
    &\hspace{6em}\int_{\cY_{(\boldsymbol{\alpha},\boldsymbol{\alpha}^\prime,\boldsymbol{\alpha}^{\prime \prime})}}
    f(\bx_{(1:r-t)})f(\bx_{(r-t+1:2r-2t)})g(\bx_{(2r-2t+1,2r)})~d\mu^{\otimes 2r}\\
    =&\quad
    \sum_{\boldsymbol{\alpha}^{\prime \prime} \in \cJ(2t)}
    \frac{\int_{\cY_{\boldsymbol{\alpha}^{\prime \prime}}}
    g(\bx_{(1:2t)})~d\mu^{\otimes 2t}}{\mu^{\otimes 2t}(\cY_\alpha)} \cdot\\
    &\hspace{6em}
    \sum_{\boldsymbol{\alpha}, \boldsymbol{\alpha}^\prime \in \cJ(r-t)}
    \ell_{\biy}(\pi^{\boldsymbol{\beta}(\boldsymbol{\alpha},\boldsymbol{\alpha}^\prime,\boldsymbol{\alpha}^{\prime \prime})})\cdot \frac{\int_{\cY_{\boldsymbol{\alpha}}}f~d\mu^{\otimes r-t}}{\mu^{\otimes r-t}(\cY_{\boldsymbol{\alpha}})}\frac{\int_{\cY_{\boldsymbol{\alpha}^\prime}}f~d\mu^{\otimes r-t}}{\mu^{\otimes r-t}(\cY_{\boldsymbol{\alpha}^\prime})}\\
    =&\quad
    \sum_{\boldsymbol{\alpha}^{\prime \prime} \in \cJ(2t)}
    \frac{\int_{\cY_{\boldsymbol{\alpha}^{\prime \prime}}}
    g(\bx_{(1:2t)})~d\mu^{\otimes 2t}}{\mu^{\otimes 2t}(\cY_\alpha)} \cdot \mathbf{p}^\top \bM_{r-t}(\pi^{\boldsymbol{\beta}(\boldsymbol{\alpha}^{\prime \prime})}\biy)\mathbf{p}, 
\end{align*}
where 
\begin{displaymath}
    \mathbf{p} = (p_{\boldsymbol{\beta}})_{\boldsymbol{\beta} \in \NN^{|\cJ|}_{r-t}},\qquad p_{\boldsymbol{\beta}} :=\sum_{\substack{\boldsymbol{\alpha} \in \cJ(r-t)\\ \boldsymbol{\beta}(\boldsymbol{\alpha})=\boldsymbol{\beta}}} \frac{\int_{\cY_{\boldsymbol{\alpha}}}f~d\mu^{\otimes r-t}}{\mu^{\otimes r-t}(\cY_{\boldsymbol{\alpha}})} \quad \forall\, \boldsymbol{\beta} \in \NN^{|\cJ|}_{r-t}.
\end{displaymath}
As a result, $\mathrm{B}_{t,g}(f,f) \geq 0$ follows from the two facts. First, the matrix 
\begin{displaymath}
    \bM_{r-t}(\pi^{\boldsymbol{\beta}(\boldsymbol{\alpha}^{\prime \prime})}\biy) \succeq 0 \quad \Rightarrow \quad \mathbf{p}^\top \bM_{r-t}(\pi^{\boldsymbol{\beta}(\boldsymbol{\alpha}^{\prime \prime})}\biy)\mathbf{p} \geq 0 \qquad \forall\, \boldsymbol{\alpha}^{\prime \prime} \in \cJ(2t).
\end{displaymath}
Second, since $g \in \mathfrak{B}_b^+(\cX^{2t})$, the integral 
\begin{displaymath}
    \frac{\int_{\cY_{\boldsymbol{\alpha}^{\prime \prime}}}
    g(\bx_{(1:2t)})~d\mu^{\otimes 2t}}{\mu^{\otimes 2t}(\cY_\alpha)}  \;\geq\; 0 \qquad \forall\, \boldsymbol{\alpha}^{\prime \prime} \in \cJ(2t).
\end{displaymath}
This establishes \eqref{eq:constraint_psd}, which completes the proof.
\end{proof}
The conclusion of this step is summarized in the following diagram, in which we observe that the convergence is verified in discretized problems. 
\begin{center}
\begin{tikzcd}[column sep=4cm, row sep=huge]
\rho^{(r)} \arrow[r, "\text{r-th relaxation } "] \arrow[d, "\text{$(\varepsilon,\delta)$--discretization }"'] & \rho^* \arrow[d, "\text{$(\varepsilon,\delta)$--discretization } "] \\
\rho^{(r)}(\varepsilon,\delta) \arrow[r, "\text{$\lim_{r \to \infty}\rho^{(r)}(\varepsilon,\delta)=\rho(\varepsilon,\delta)$}"] & \rho(\varepsilon,\delta)
\end{tikzcd}  
\end{center}

\subsection{Step 3: Convergence of\texorpdfstring{~\eqref{eq:ot-r-th}}{(1)} to\texorpdfstring{~\eqref{eq:ot_quadform}}{(1)}}
Our next result controls how closed the discretized problems~\eqref{eq:discrete_qp} and~\eqref{eq:discrete_convexp} approximate~\eqref{eq:ot_quadform} and~\eqref{eq:ot-r-th} respectively by parameters $\varepsilon$ and $\delta$. These tightness are presented in the following lemmas.

\begin{lemma}\label{lemma: discrete of rho}
Given positive $\varepsilon,\delta>0$, we have the following inequality:
\begin{equation}\label{eq:tightness of discrete qp}
|\ropt - \rho(\varepsilon,\delta) | \leq L(\varepsilon)d^{(2)}(\delta) + \varepsilon.
\end{equation}
\end{lemma}

\begin{proof}
For any $\pi \in \Pi$, we define $\overline{\pi} = (\pi_{\alpha} := \pi(\cY_{\alpha}))_{\alpha \in \cJ}$.  Then by a straightforward calculation, one can show that the marginal conditions on $\pi$ imply $\overline{\pi}$ belongs to $\Delta(\varepsilon,\delta)$.  Then using the assumption that $c$ is $L(\varepsilon)$-Lipschitz continuous over $\cY$ as well as the uniform tightness assumption from Assumption \ref{assume:tightness}, we have the following inequalities
\begin{equation*}
\begin{aligned}
& \Biggl| \int_{\cX^2}c(\bx_1,\bx_2)~d\pi(\bx_1)d\pi(\bx_2) - \langle c(\varepsilon,\delta), \overline{\pi} \otimes \overline{\pi} \rangle \Biggr|  \\
\leq~ & \Biggl| \int_{\cX^2}c(\bx_1,\bx_2)~d\pi(\bx_1)d\pi(\bx_2) - \int_{\cY^2}c(\bx_1,\bx_2)~d\pi(\bx_1)d\pi(\bx_2) \Biggr|+ \\
& \qquad \qquad \qquad \qquad  \Biggl| \int_{\cY^2}c(\bx_1,\bx_2)~d\pi(\bx_1)d\pi(\bx_2) - \langle c(\varepsilon,\delta), \overline{\pi} \otimes \overline{\pi} \rangle \Biggr| \\
\leq~ & L(\varepsilon)d^{(2)}(\delta) + \varepsilon.
\end{aligned}
\end{equation*}
Hence, we have 
\begin{displaymath}
    \int_{\cX^2}c(\bx_1,\bx_2)~d\pi(\bx_1)d\pi(\bx_2) +  L(\varepsilon)d^{(2)}(\delta) + \varepsilon \geq 
    \langle c(\varepsilon,\delta), \overline{\pi} \otimes \overline{\pi} \rangle,
\end{displaymath}
and we take infimum of the left hand side over $\pi \in \Pi$ to obtain 
\begin{equation}\label{eq:discrete of rho 1}
    \rho + L(\varepsilon)d^{(2)}(\delta) + \varepsilon \geq \rho(\varepsilon,\delta).
\end{equation}

Next, for any $\overline{\pi} = (\pi_{\alpha} := \pi(\cY_{\alpha}))_{\alpha \in \cJ} \in \Delta(\varepsilon,\delta)$, we define 
\begin{displaymath}
    \pi = \sum_{\alpha \in \cJ} \frac{\overline{\pi}_{\alpha}}{\mu(\cY_\alpha)}\cdot \mu\big|_{\cY_{\alpha}}.
\end{displaymath}
By direct calculation, we can prove that $\pi \in \Pi$ similarly to the proof for marginal condition in Theorem~\ref{thm: connnection between relaxations}. Moreover, we also have 
\begin{align*}
&\Bigl| \int_{\cX^2}c(\bx_1,\bx_2)~d\pi(\bx_1)d\pi(\bx_2) - \langle c(\varepsilon,\delta), \overline{\pi} \otimes \overline{\pi} \rangle \Bigr| \leq L(\varepsilon)d^{(2)}(\delta) + \varepsilon\\
\Rightarrow\quad & \langle c(\varepsilon,\delta), \overline{\pi} \otimes \overline{\pi} \rangle + L(\varepsilon)d^{(2)}(\delta) + \varepsilon \geq \int_{\cX^2}c(\bx_1,\bx_2)~d\pi(\bx_1)d\pi(\bx_2),
\end{align*}
for which we take infimum of the left hand side over $\overline{\pi} \in \Delta(\varepsilon,\delta)$ to obtain that 
\begin{equation}\label{eq:discrete of rho 2}
    \rho(\varepsilon,\delta) + L(\varepsilon)d^{(2)}(\delta) + \varepsilon \geq \rho.
\end{equation}
The lemma follows from~\eqref{eq:discrete of rho 1} and~\eqref{eq:discrete of rho 2}.
\end{proof}

\begin{lemma}\label{lemma: discrete of rho_r}
Given positive $\varepsilon,\delta>0$, we have the following inequality:
\begin{equation}\label{eq:tightness of discrete convexp}
|\rho^{(r)} - \rho^{(r)}(\varepsilon,\delta) | \leq L(\varepsilon)d^{(2)}(\delta) + \varepsilon.
\end{equation}
\end{lemma}
\begin{proof}
    The idea to prove this lemma is similar to Lemma~\eqref{lemma: discrete of rho}. First, for any $\bP \in \Pi_r$, we recall that $\biy(\bP)$ is a feasible solution of~\eqref{eq:discrete_convexp}, and the inequality~\eqref{eq:discrete_approx_err} shows that 
    \begin{displaymath}
        \left|\int_{\cX^{2r}}c(\bx_{(1,2)})~d\bP- \sum_{\boldsymbol{\alpha} \in \cI(2r)}c(\bx_{\alpha,(1,2)})\bP_{\boldsymbol{\alpha}}\right| \leq L(\varepsilon)d^{(2)}(\delta) + \varepsilon,
    \end{displaymath}
    where $\sum_{\boldsymbol{\alpha} \in \cI(2r)}c(\bx_{\alpha,(1,2)})\bP_{\boldsymbol{\alpha}} = \langle \bc(\varepsilon,\delta),\biy(\bP)\rangle$, Therefore, we have 
    \begin{displaymath}
        \int_{\cX^{2r}}c(\bx_{(1,2)})~d\bP + L(\varepsilon)d^{(2)}(\delta) + \varepsilon \geq \langle \bc(\varepsilon,\delta),\biy(\bP)\rangle,
    \end{displaymath}
    in which we take infimum of the left hand side over $\bP \in \Pi_r$ to obtain that 
    \begin{equation}\label{eq:discrete of rho_r1}
        \rho^{(r)} + L(\varepsilon)d^{(2)}(\delta) + \varepsilon \geq \rho^{(r)}(\varepsilon,\delta)
    \end{equation}

    Second, for any $\biy \in \biy(\Pi_r)$, Theorem~\ref{thm: connnection between relaxations} shows that $\bP(\biy) \in \Pi_r$, which combine with the inequality~\eqref{eq:discrete_approx_err} to induce that 
    \begin{displaymath}
        \varepsilon \geq \langle \bc(\varepsilon,\delta),\biy\rangle + L(\varepsilon)d^{(2)}(\delta) + \varepsilon \geq \int_{\cX^{2r}}c(\bx_{(1,2)})~d\bP(\biy),
    \end{displaymath}
    and take infimum of the left hand side over $\biy \in \biy(\Pi)$, the following inequality holds 
    \begin{equation}\label{eq:discrete of rho_r2}
        \rho^{(r)}(\varepsilon,\delta) + L(\varepsilon)d^{(2)}(\delta) + \varepsilon \geq \rho^{(r)}.
    \end{equation}
    As a result, the lemma follows from~\eqref{eq:discrete of rho_r1} and~\eqref{eq:discrete of rho_r2}.
\end{proof}

We now have all ingredients to prove Theorem \ref{thm:convergence_tightness}.
\begin{proof} [Proof of Theorem \ref{thm:convergence_tightness}]
 From Lemma~\ref{lemma: discrete of rho} and Lemma~\ref{lemma: discrete of rho_r} we have the estimate
\begin{equation*}
    |\ropt - \rho(\varepsilon,\delta) | \leq L(\varepsilon)d^{(2)}(\delta) + \varepsilon,
    \quad \mbox{and} \quad
    |\rho^{(r)} - \rho^{(r)}(\varepsilon,\delta)| \leq L(\varepsilon)d^{(2)}(\delta) + \varepsilon.
\end{equation*}
By combining the above bound with the triangle inequality we have
\begin{equation*}
\begin{aligned}
    \ropt - \rho^{(r)} ~\leq~ & |\ropt - \rho(\varepsilon,\delta)| + |\rho(\varepsilon,\delta) -\rho^{(r)}(\varepsilon,\delta)| + |\rho^{(r)}(\varepsilon,\delta) -\rho^{(r)}| \\
    \leq~ & 2(L(\varepsilon)d^{(2)}(\delta) + \varepsilon) + |\rho(\varepsilon,\delta) - \rho^{(r)}(\varepsilon,\delta)|.
\end{aligned}
\end{equation*}
Then, by taking the limit $r \rightarrow \infty$, we obtain
\begin{equation*}
    \ropt - \lim_{r \to \infty}\rho^{(r)} ~\leq~ 2(L(\varepsilon)d^{(2)}(\delta)+\varepsilon) + \lim_{r \to \infty} |\rho(\varepsilon,\delta) - \rho^{(r)}(\varepsilon,\delta)| ~ =~ 2(L(\varepsilon)d^{(2)}(\delta)+\varepsilon).
\end{equation*}
Note that the LHS is independent of $\varepsilon$ and $\delta$.  To conclude the result, we show that it is possible to pick a sequence $\{ (\varepsilon_j, \delta_j) \}_{j=1}^{\infty}$ such that the RHS vanishes to zero.  Suppose we pick  $\varepsilon_j = 1/2^j$.  For each $j$, we make a choice of $\mathcal{Y}(\varepsilon_j)$ satisfying the tightness condition \eqref{eq: tightness}, as described earlier.  Each choice of $\mathcal{Y}(\varepsilon_j)$ induces a corresponding Lipschitz bound $L(\varepsilon_j)$.  Now we pick $\delta_j$ sufficiently small so that $d^{(2)} (\delta_j) \leq 1/(2^j L(\varepsilon_j))$.  We then have that the RHS is at most $2(L(\varepsilon_j)d^{(2)}(\delta_j)+\varepsilon_j) \leq 1/2^{j-2}$, which vanishes to zero as we take $j \rightarrow \infty$.  This implies $\lim_{r \to \infty}\rho^{(r)} = \ropt$.
\end{proof}

\section{Moment-type SDP relaxation}\label{sec: moment_hierarchy}

We return our focus to the optimization instance \eqref{eq:ot-r-th}.  As a reminder, this takes place over the space of probability measures.  In this section, we describe a framework for computing the optimal value to \eqref{eq:ot-r-th} via {\em sequences of moments} associated with the problem.  Our approach builds on the ideas developed in \cite{MN:24} in which the moment-SOS hierarchy was applied to problems in optimal transport.  In particular, through our discussion, we will describe a hierarchy of finite dimensional semidefinite programming relaxations whose optimal solution increasingly approximate that of \eqref{eq:ot-r-th}.  

The development of our moment-SOS hierarchy will require  additional assumptions to the formulation \eqref{eq:ot-r-th}.  First, it is {\em essential} that the metric spaces $\cX_j$ are Euclidean.  As a note, although we assume that the $\cX_j$'s are Euclidean from the beginning of the paper, much of the development in the previous section extends to general metric spaces, and the corresponding analysis has been written to reflect this point.  The framework we introduce in the current section, however, relies on $\cX_j$ being Euclidean in a critical way.  Second, we assume that the cost function $c(\bx_1,\bx_2)$ is a {\em polynomial} function in the input variables $\bx_1,\bx_2$.  These assumptions allow us to express the optimization problem in terms of moment sequences.  In the remainder of this section, we use the notation $\bx_i = (\bx_{i,1},\dots,\bx_{i,m}) \in \cX$, where $\bx_{i,j} \in \cX_j \; \forall j \in [m]$.

We briefly mention a practical example where the above-stated assumptions do hold.  Consider the task of computing the Gromov-Wasserstein distance between two distributions residing in Euclidean space.  A common choice of the distance function is the squared Euclidean loss, and a common choice of the distortion between metric evaluations is the squared loss, in which case the objective is polynomial in $\cX$ (it is in fact quartic).   

\subsection{From probability measures to moment sequences}

As a reminder, the main decision variable in \eqref{eq:ot-r-th} is a probability measure $\bP \in \mathcal{P}(\cX^{2r})$.  Under suitable conditions, a probability measure $\mu \in \mathcal{P}(\cX)$, $\cX \subset \RR^n$, can be represented via its moment sequence $\{\biy_{\boldsymbol{\alpha}}\}_{\alpha \in \NN^n}$, where
\begin{equation*}
    \biy_{\boldsymbol{\alpha}} = \int_{\cX} \bx^{\alpha}~d\mu.
\end{equation*}
In the following, the moment sequence $\biy_{\boldsymbol{\alpha}}$ will be our new decision variable.  That is to say, we will treat the values of these moments as decision variables, and we optimize over these values.  As such, our immediate task is to express the constraints in \eqref{eq:constraint_sym}, \eqref{eq:constraint_mar}, and \eqref{eq:constraint_psd} as constraints in $\biy_{\boldsymbol{\alpha}}$.

\begin{lemma} [Symmetry condition] \label{thm:sym_moments}
   Let $\{\biy_{\boldsymbol{\alpha}}\}_{\boldsymbol{\alpha} \in \NN^{2rn}}$ be a moment sequence representing a probability measure $\bP$ supported on $\cX^{2r}$.  Then $\bP$ satisfies the symmetric condition~\eqref{eq:constraint_sym} iff for any $\boldsymbol{\alpha} = (\alpha_1,\dots,\alpha_{2r}) \in \NN^{2rn}$ with $\alpha_i \in \NN^n \; \forall i \in [2r]$, and any permutation $\sigma \in S_{2r}$, the following holds 
   \begin{equation}\label{eq:constraint_sym_moment}
       \biy_{(\alpha_1,\dots,\alpha_{2r})} =  \biy_{(\alpha_{\sigma(1)},\dots,\alpha_{\sigma(2r)})}.
   \end{equation}
\end{lemma}

\begin{lemma} [Marginal condition] \label{thm:mar_moments} Let $\{\biy_{\boldsymbol{\alpha}}\}_{\boldsymbol{\alpha} \in \NN^{2rn}}$ be a moment sequence representing a probability measure $\bP$ supported on $\cX^{2r}$.  Then $\bP$ satisfies the marginal condition~\eqref{eq:constraint_mar} iff for any $\boldsymbol{\alpha} = (\alpha_1,\dots,\alpha_{2r}) \in \NN^{2rn}$ with $\alpha_i \in \NN^n \; \forall i \in [2r]$, and for any $j \in [2r]$, the following holds
\begin{equation}\label{eq:constraint_mar_moment}
    \biy_{(\alpha_1,\dots,\alpha_{j,1},(0,\dots,\alpha_{j,k},\dots,0),\dots,\alpha_{2r})}=m_{\alpha_{j,k}}(\mu_k)\cdot \biy_{(\alpha_1,\dots,\alpha_{j,1},0,\dots,\alpha_{2r})},
\end{equation}
where $\alpha_{j,k} \in \NN^{n_j}$ and the moment $m_{\alpha_{j,k}}(\mu_k) = \int_{\cX_k} \bx^{\alpha_{j,k}}~d\mu_k$ is already defined based on the marginal distribution $\mu_k$.
\end{lemma}

The proofs of Lemmas \ref{thm:sym_moments} and \ref{thm:mar_moments} are straightforward calculations, and we omit these.  The conditions \eqref{eq:constraint_sym_moment} and \eqref{eq:constraint_mar_moment} can be simplified if they are required to hold simultaneously:
\begin{equation*}
\begin{aligned}
    & \biy_{(\alpha_1,\dots,\alpha_{2r})} =  \biy_{(\alpha_{\sigma(1)},\dots,\alpha_{\sigma(2r)})} \; \forall (\alpha_1,\dots,\alpha_{2r}) \in \NN^{2rn},\; \sigma \in S_{2r},\\
    & \biy_{((0,\dots,\alpha_{1,k},\dots,0),\dots,\alpha_{2r})}=m_{\alpha_{1,k}}(\mu_k)\cdot \biy_{(0,\dots,\alpha_{2r})}\; \forall k \in [m],\; \alpha_i \in \NN^n,\; \alpha_{1,k} \in \NN^{n_k}.
\end{aligned}
\end{equation*}

The condition \eqref{eq:constraint_psd} is more considerably more delicate.  In the following, we consider two different settings:  In the first, we assume that the underlying distributions satisfy a very permissive requirement known as the generalized Carleman's condition, under which we derive a moment-SOS hierarchy.  In the second, we assume that the sets $\cX_j$ are semi-algebraic.  Using the polynomials that specify $\cX_j$, we derive a correspondingly tighter hierarchy of relaxations. 

\subsection{Generalized Carleman's condition}

The generalized Carleman's condition describe conditions under which a moment sequence specifies a probability measure.  
Concretely, a measure $\mu \in \mathcal{P} (\RR^n)$ is said to satisfy the generalized Carleman's condition if the following holds for all $i \in [n]$:
\begin{equation*}
    \sum_{n=1}^\infty \left(\int x_i^{2n} \right)^{-\frac{1}{2n}}~d\mu(\bx) = +\infty.
\end{equation*}
The Carleman's condition is satisfied by a large family of distributions.  For instance, it contains the gaussian distribution, as well as probability distributions that are compactly supported. 
%It also includes the case when $\cX$ is compact and the marginal measures are probability measures. Assume that $\cX$ is contained in a ball of radius $R$, then for any $i$, we evaluate 
%\begin{equation*}
%    \int_{\cX} x_i^n \leq R^{2n} \; \Rightarrow \sum_{i=1}^\infty \Bigl(\int_\cX x_i^{2n} \Bigr)^{\frac{-1}{2n}} \;\geq\; \sum_{i=1}^\infty \frac{1}{R} = \infty.
%\end{equation*}

We make the following assumption:
\begin{assume}\label{assume: Carleman}
    All the marginal measures $\{\mu_j\}_{j=1}^{m}$ satisfy the generalized Carleman's condition.
\end{assume}
As it turns out, this suffices to imply the generalized Carleman's condition on the space of feasible probability measures on which we optimize over:
\begin{lemma}
    Under Assumption~\ref{assume: Carleman}, all feasible measures $\bP \in \Pi_r$ also satisfy the generalized Carleman's condition.
\end{lemma}

\begin{proof}
    Given $\bx \in \cX^{2r}$, we write (i) $\bx = (\bx_1,\dots,\bx_{2r}) \in \cX^{2r}$, where (ii) $\bx_i =(\bx_{i,1},\dots,\bx_{i,m})\forall i \in [2r]$, and where (iii) $\bx_{i,j} = (\bx_{i,j,1},\dots,\bx_{i,j,n_j})$, with $n_j = \dim \cX_j$.  Using this notation, by applying Assumption~\ref{assume: Carleman}, we have that
    \begin{equation*}
        \sum_{n=1}^\infty \left(\int x_{i,j,k}^{2n} \right)^{-\frac{1}{2n}}~d\mu(\bx) = \sum_{n=1}^\infty \left(\int_{\cX_j} x_k^{2n} \right)^{-\frac{1}{2n}}~d\mu(\bx) = +\infty.
    \end{equation*}
\end{proof}

There are two fundamental consequences of the generalized Carleman condition that underlie our approach. First, a probability measure satisfying this condition is uniquely determined by its moment sequence. Second, for every $1\leq p<\infty$, the space of polynomials is dense in $L^p(\bP)$; see~\cite[Theorem~2.3]{De:2003}. The latter property further implies that the cone of sum-of-squares polynomials is dense in the nonnegative function cone of $L^p(\bP)$.

Indeed, fix $1\leq p<\infty$ and define
\begin{displaymath}
    L_+^p(\bP)
    :=
    \bigl\{
        g\in L^p(\bP):
        g\geq0\quad \bP\text{-a.e.}
    \bigr\}.
\end{displaymath}
By~\cite[Theorem~2.3]{De:2003}, the generalized Carleman condition implies that $\RR[\bx]$ is dense in $L^{2p}(\bP)$. Let $g\in L_+^p(\bP)$. Since
\begin{displaymath}
    \int \left|\sqrt{g}\right|^{2p}\,d\bP
    =
    \int |g|^p\,d\bP
    <\infty,
\end{displaymath}
we have $\sqrt{g}\in L^{2p}(\bP)$. Hence, there exists a sequence of polynomials $\{g_k\}_{k\in\NN}\subset\RR[\bx]$ such that
\begin{displaymath}
    \lim_{k\to\infty}
    \left\|
        g_k-\sqrt{g}
    \right\|_{L^{2p}(\bP)}
    =
    0.
\end{displaymath}
Since $g_k^2\in\Sigma[\bx]$, H\"older's inequality yields
\begin{align*}
    \left\|g_k^2-g\right\|_{L^p(\bP)}
    &=
    \left\|
        (g_k-\sqrt{g})(g_k+\sqrt{g})
    \right\|_{L^p(\bP)}
    \\
    &\leq
    \left\|g_k-\sqrt{g}\right\|_{L^{2p}(\bP)}
    \left\|g_k+\sqrt{g}\right\|_{L^{2p}(\bP)}.
\end{align*}
Moreover,
\begin{displaymath}
    \left\|g_k+\sqrt{g}\right\|_{L^{2p}(\bP)}
    \leq
    \left\|g_k-\sqrt{g}\right\|_{L^{2p}(\bP)}
    +
    2\left\|\sqrt{g}\right\|_{L^{2p}(\bP)},
\end{displaymath}
and therefore the sequence
$\{\|g_k+\sqrt{g}\|_{L^{2p}(\bP)}\}_{k\in\NN}$ is bounded. It follows that
\begin{displaymath}
    \lim_{k\to\infty}
    \left\|g_k^2-g\right\|_{L^p(\bP)}
    =
    0.
\end{displaymath}
Consequently,
\begin{displaymath}
    \overline{\Sigma[\bx]}^{\,L^p(\bP)}
    =
    L_+^p(\bP).
\end{displaymath}
We use these two consequences of the generalized Carleman condition to reformulate~\eqref{eq:ot-r-th} as an optimization problem over moment sequences, stated in the following theorems.
\begin{theorem} \label{thm:conversion_carleman}
    Suppose Assumption~\ref{assume: Carleman} holds.  Then a moment sequence $\{\biy_{\boldsymbol{\alpha}}\}_{\boldsymbol{\alpha} \in \NN^{2rn}}$ admits a feasible measure $\bP \in \Pi_r$ of~\eqref{eq:ot-r-th} iff the following conditions are satisfied simultaneously: 
    \begin{enumerate}
        \item $\biy_{(\alpha_1,\dots,\alpha_{2r})} =  \biy_{(\alpha_{\sigma(1)},\dots,\alpha_{\sigma(2r)})} \; \forall (\alpha_1,\dots,\alpha_{2r}) \in \NN^{2rn},\; \sigma \in S_{2r}$.
        \item $\biy_{((0,\dots,\alpha_{1,k},\dots,0),\dots,\alpha_{2r})}=m_{\alpha_{1,k}}(\mu_k)\cdot \biy_{(0,\dots,\alpha_{2r})}\; \forall k \in [m],\; \alpha_i \in \NN^n,\; \alpha_{1,k} \in \NN^{n_k}$.
        \item $\bM^r_d(\biy) \succeq 0$, $\bQ^{r,t}_{d,l}(\biy) \succeq 0$ for all $0 \leq t \leq r$, and $d, l \in \NN$.  Here, $\bM^r_d(\biy)$ denotes the $d$--truncated moment matrix, defined by 
        \begin{displaymath}
            \bM^r_d(\biy):= \ell_{\biy}(\bv_d(\bx)\bv_d(\bx)^\top),\quad \mbox{note that $\bx \in \RR^{2rn}$}.
        \end{displaymath}
        and $\bQ^{r,t}_{d,l}(\biy)$ is defined by 
        \begin{multline*}
            \bQ^{r,t}_{d,l}(\biy) = \ell_y \big([\bv_d(\bx_{(1,r-t)})\bv_d(\bx_{(r-t+1,2r-2t)})^\top]
            \otimes [\bv_l(\bx_{(2r-2t+1,2r)})\bv_l(\bx_{(2r-2t+1,2r)})^\top] \big).
        \end{multline*}
    \end{enumerate}
    In particular, suppose we let $\cM^r$ denote the subset of moment sequences that satisfy the above conditions.  Then one can reformulate~\eqref{eq:ot-r-th} as 
\begin{equation*}
    \rho^r ~~:=~~ \inf \; \ell_y(c) \quad \mbox{subject to } \quad y \in \cM^r.
\end{equation*}
\end{theorem}

\begin{proof}[Proof of Theorem \ref{thm:conversion_carleman}]
     Based on the marginal condition, we have 
     \begin{equation*}
         \sum_{k=1}^\infty[\ell_{\biy}(\bx_{i,j,l}^{2k})]^{\frac{-1}{2k}}= \sum_{k=1}^\infty \Bigl[\int_{\cX_j}x_l^{2k}\Bigr]^{\frac{-1}{2k}}=\infty \quad \forall i \in [2r],\; j \in [m], k \in [n_j].
     \end{equation*}
     This, combined with the condition that $\bM^r_d(\biy) \succeq 0$, implies that there exists a probability $\bP \in \mathcal{P}(\cX^{2r})$ such that the sequence $\biy$ represents the measure $\bP$ \cite[Proposition 3.5]{Las:09}.  Furthermore, because $\bP$ satisfies the Carleman's condition, we recall the following consequences:
     \begin{enumerate}
         \item The set of polynomials is dense in $L^p(\bP)$ for any $1\leq p < \infty$
         \item The set of SOS polynomials is dense in the cone of non-negative functions with respect to $L^p(\bP)$ for any $1 \leq p < \infty$. 
\end{enumerate}
     Recall that  the PSD condition of $\bP$ is expressed as
     \begin{multline*}  \int_{\cX^{2r}}f(\bx_{(1,r-t)})f(\bx_{(r-t+1,2r-2t)})g(\bx_{(2r-2t+1,2r)})~d\bP \geq 0 \\ \forall\quad 0 \leq t \leq r,\; f \in \mathfrak{B}_b(\cX^{r-t}) \subset L^p(\bP),\; g \in \mathfrak{B}_b^+(\cX^{r-t}) \subset L^p_+(\bP).
     \end{multline*}
     We can choose $p=3$, then there exist a sequences of polynomial $\{f_k\}$ converging to $f$ in $L^3(\bP)$-norm and a sequence of SOS polynomial $\{g_k\}$ converging to $g$ in $L^3(\bP)$-norm. Then by H\"older's inequality yields
     \begin{multline*}
         \lim_{k \to \infty} \int_{\cX^{2r}}f_k(\bx_{(1,r-t)})f_k(\bx_{(r-t+1,2r-2t)})g_k(\bx_{(2r-2t+1,2r)})~d\bP \\
         = \int_{\cX^{2r}}f(\bx_{(1,r-t)})f(\bx_{(r-t+1,2r-2t)})g(\bx_{(2r-2t+1,2r)})~d\bP.
     \end{multline*}
     Hence, the PSD condition can be expressed as 
    \begin{multline*}  
    \int_{\cX^{2r}}f(\bx_{(1,r-t)})f(\bx_{(r-t+1,2r-2t)})g(\bx_{(2r-2t+1,2r)})~d\bP \geq 0 \\ \forall\quad 0 \leq t \leq r,\; f \in \RR[\bx_{(1,r-t)},\; g \in \Sigma[\bx_{(1,2t)}].
     \end{multline*} 
     Let $\mathbf{f}$ be the vector of coefficients of the polynomials $f$ (of degree $d$). Let $g$ be a SOS polynomial of degree $2l$, then there exists a PSD matrix $G$ such that
     \begin{displaymath}
         g(\bx_{(2r-2t+1,2r)}) = \langle G, \bv_l(\bx_{(2r-2t+1,2r)})\bv_l(\bx_{(2r-2t+1,2r)})^\top \rangle,\quad G = \sum_k \mathbf{g}_k\mathbf{g}_k^\top.
     \end{displaymath}
     This allows us to represent the integral with the Riesz functionals of $f$  and $g$ as follows: 
     \begin{align*}
         &\quad \int_{\cX^{2r}}f(\bx_{(1,r-t)})f(\bx_{(r-t+1,2r-2t)})g(\bx_{(2r-2t+1,2r)})~d\bP \\
         =&\quad \int_{\cX^{2r}}\langle\mathbf{f}\mathbf{f}^\top, \bv_d(\bx_{(1,r-t)})\bv_d(\bx_{(r-t+1,2r-2t)})^\top \rangle \cdot \langle 
         \langle G, \bv_l(\bx_{(2r-2t+1,2r)})\bv_l(\bx_{(2r-2t+1,2r)})^\top \rangle~d\bP \\
         =&\quad \int_{\cX^{2r}}\Biggl\langle \left(\sum_k\mathbf{f}\otimes\mathbf{g}_k\right)\left(\sum_k\mathbf{f}\otimes\mathbf{g}_k\right)^\top, [\bv_d(\bx_{(1,r-t)})\bv_d(\bx_{(r-t+1,2r-2t)})^\top]\\
         &\hspace{8em}\otimes [\bv_l(\bx_{(2r-2t+1,2t)})\bv_l(\bx_{(2r-2t+1,2t)})^\top)] \Biggr\rangle~d\bP \\
         =&\quad \left\langle \left(\sum_k\mathbf{f}\otimes\mathbf{g}_k\right)\left(\sum_k\mathbf{f}\otimes\mathbf{g}_k\right)^\top, \bQ^{r,t}_{d,l}(\biy) \right\rangle.
     \end{align*}
     Consequently, the requirement that $\bP$ satisfies the PSD condition \eqref{eq:constraint_psd} is equivalent to the requirement that
     \begin{equation*}
         \bM_d^r(\biy) \succeq 0, \quad \text{and}\quad \cQ^{r,t}_{d,l}(\biy) \succeq 0 \quad \forall\, ) \leq t \leq r,\, d,\,l \in \NN.
     \end{equation*}
\end{proof}

By restricting moment sequences to monomials of bounded degree, Theorem \ref{thm:conversion_carleman} suggests the following SDP hierarchy to solve~\eqref{eq:ot-r-th}.

    Given an integer $d$ that $2d \geq \deg(c)$, let $\cM^r_d$ be the set of truncated pseudo moment sequences $\{\biy_{\boldsymbol{\alpha}}\}_{\boldsymbol{\alpha} \in \NN^{2rn}_{2d}}$ satisfying the following conditions: 
    \begin{enumerate}
        \item $\biy_{(\alpha_1,\dots,\alpha_{2r})} =  \biy_{(\alpha_{\sigma(1)},\dots,\alpha_{\sigma(2r)})} \; \forall (\alpha_1,\dots,\alpha_{2r}) \in \NN^{2rn},\; \sigma \in S_{2r}$,
        \item $\biy_{((0,\dots,\alpha_{1,k},\dots,0),\dots,\alpha_{2r})}=m_{\alpha_{1,k}}(\mu_k)\cdot \biy_{(0,\dots,\alpha_{2r})}\; \forall k \in [m],\; \alpha_i \in \NN^n,\; \alpha_{1,k} \in \NN^{n_k}$,
        \item $\bM^r_d(\biy) \succeq 0$, and $\bQ^{r,t}_{p,q}(\biy) \succeq 0 \; \forall p + q \leq d,\; 0 \leq t \leq r$.
    \end{enumerate}
    Consequently, we have the following SDP relaxation 
    \begin{equation}\label{eq: double hierarchy}
        \rho^r_d :=\; \inf\; \ell_y(c) \quad \mbox{subject to } y \in \cM^r_d.
    \end{equation}
    Furthermore, one has $\rho^r_d  \rightarrow \rho$ as $\min \{r,d \} \rightarrow \infty$.

\subsection{Strengthened descriptions for semi-algebraic sets}\label{sec: strengthening hierarchy}

In certain settings, one may be given further information that the spaces $\cX_j$'s are semialgebraic, and whose descriptions are available to us.  In this section, we explain how these descriptions help us obtain semidefinite relaxations that strengthen \eqref{eq: double hierarchy}.

More concretely, suppose all $j \in [m]$ that the sets $\cX_j$ are semialgebraic and whose descriptions are:
\begin{equation*}
    \cX_j = \{\bx : g_{j,k}(\bx) \geq 0 \; \forall k \in [N_j]\} \subseteq \mathbb{R}^{n_j}.
\end{equation*}
Here, the $g_{j,k}$'s are polynomials, while $N_j$ is the number of polynomial required to specify $\cX_j$.  Given a polynomial $f$, let $\lceil f \rceil := \lceil \deg f /2\rceil$, and denote $d_{\min} := \max\{\lceil c \rceil, \lceil g_{j,k} \rceil \; \forall j \in [m],\; k \in [N_j]\}$.  In the following, we introduce two different classes of conditions that each lead to its own set of additional (valid) constraints, which one may include within the above SDP relaxations to further strengthen these relaxations. 

% As a reminder, these conditions are defined on $\overline{\biy} \in (\NN^n)^{2r}_{2d}$, and we define $\ell_{\biy} = \ell_{\overline{\biy}}$, where $\overline{\biy}$ is the pseudo-moment sequence constructed from $\biy$ based on the symmetry structure.

\textbf{Localizing matrices.} One classical approach is to introduce additional PSD constraints associated to the 
\begin{equation*}
    \bM^r_{d- \lceil g_{j,k} \rceil}(g_{j,k}\biy)= \ell_{\biy}(g_{j,k}\bv_{d- \lceil g_{j,k} \rceil}(\bx)\bv_{d- \lceil g_{j,k} \rceil}(\bx)^\top) \succeq 0 \quad \forall\, j \in [m],\; k\in [N_j].
\end{equation*}

\textbf{Approximation of non-negative polynomial.}  In the previous section, we have made an argument that the set of SOS polynomials is dense in the set of non-negative polynomials in $L^2$--sense. However, when the set $\cX$ is a semi-algebraic set, then we can improve the approximation of non-negative polynomials on $\cX$. For example, when solving for $\rho^{(r)}$, we need to consider the semi-algebraic set 
\begin{equation*}
    \cX^{2r} = \{g_{j,k}(\bx_{l}) \geq 0:\; \forall j \in [m],\; k \in [N_j],\; j \in [2r] \}. 
\end{equation*}
where $\cX^{2r} \ni \bx = (\bx_1,\dots,\bx_{2r}), \; \bx_l \in \cX \;\forall\, l \in [2r]$.

Then adding these constraints will probably increase the tightness of SDP relaxation: for any $p, q$ that $p+q = 2d$, 
\begin{multline*}
    \cQ^{r,t}_{p,q}(g_{j,k}) = \ell_{\biy}\Bigl(\bv_{p}(\bx_{(1,r-t)})\bv_p(\bx_{r-t+1,2r-2t})^\top]\\ \otimes [g_{j,k}(\bx_{2r})\bv_{q-\lceil g_{j,k}\rceil}(\bx_{(2r-2t+1,2r)}) \bv_{q-\lceil g_{j,k}\rceil}(\bx_{(2r-2t+1,2r)})^\top] \Bigr) \succeq 0.
\end{multline*}
Here, we only consider the polynomial $g_{j,k}$ on $\bx_{2r}$ because of the symmetry of the $\biy$.

\section{Numerical Experiments}  \label{sec: numerical_experiments}

In this section, we apply our approach to solve a series of quadratic optimization instances over probability measures.  Our goal is to investigate the strength of our proposed relaxation.

As a reminder, there are two parameters that govern the hierarchy:  The parameter $r$ governs the size of the product space (note that the underlying space is $\cX^{2r}$), while the parameter $d$ bounds the degree of the monomials we search over (sometimes known as the {\em truncation order}).  In our experiments, we primarily select $r=1$, while the range of values of $d$ span a larger range, which we specify in our individual experiments.  All experiments are implemented in MATLAB using the SDP modeling toolbox YALMIP~\cite{Lof:2004} with MOSEK as the solver~\cite{Mosek:2026}, and are run on a MacBook Pro equipped with an Apple M1 chip, a 10-core CPU, and 16 GB of RAM.

\subsection{Gromov-Wasserstein distance between Gaussian distributions}

In the first series of experiments, we apply our technique to compute the Gromov-Wasserstein distance between pairs of Gaussian distributions. Let $\mu_{p} \sim \mathcal{N}(\mathbf{m}_p, \Sigma_p)$ and $\mu_{q} \sim \mathcal{N}(\mathbf{m}_q,\Sigma_q)$ be Gaussian distributions on $\RR^m$ and $\RR^n$, respectively.  We consider the SDP relaxation when $r=1$.  For $d \geq 2$, our proposed SDP reads 
\begin{equation*}
    \inf_{\biy \in (\NN)^{2r}_{2d}} \ell_{\biy}((\|p-p'\|_{2}^2 - \|q -q'\|_{2}^2)^2) \quad \mbox{subject to} \quad \biy \in \cM^1_d(\mu_p,\mu_q).
\end{equation*}
Here, $p,p' \in \RR^m$ and $q,q' \in \RR^n$.  In addition, $\cM^1_d (\mu_p,\mu_q)$ is indexed by the collection of multi-indices of the form $(\alpha,\beta)$, $\alpha, \beta \in \NN^{m+n}$, satisfying $|\alpha| + |\beta| \leq 2d$.  Suppose we write $\alpha = (\alpha_1,\alpha_2)$ where $\alpha_1 \in \NN^m$ and $\alpha_2 \in \NN^n$, and it is similar for $\beta$.  Then the feasible set $\cM^1_d (\mu_p,\mu_q)$ is the set of (truncated) pseudo-moment sequences $\biy$ satisfying: 
\begin{enumerate}
    \item {[Symmetric condition]:} For any $\alpha, \beta \in \NN^{m+n}$ such that $|\alpha| +|\beta| \leq 2d$, the following holds 
    \begin{displaymath}
        \biy_{(\alpha,\beta)} = \biy_{(\beta,\alpha)}. 
    \end{displaymath}
    
    \item {[Marginal condition]:} For any $\alpha_1 \in \NN^m, \alpha_2 \in \NN^n, \beta \in \NN^{m+n}$ satisfying $|\alpha_1| +|\beta| \leq 2d$ and $|\alpha_2| + |\beta| \leq 2d$, we have 
    \begin{equation*}
        \biy_{(\alpha_1,0,\beta)} = m_{\alpha_1}(\mu_p)\cdot \biy_{(0,0,\beta)} \quad \mbox{and}\quad y_{(0,\alpha_2,\beta)} = m_{\alpha_2}(\mu_q) \cdot \biy_{(0,0,\beta)}.
    \end{equation*}
    
    \item {[PSD conditions]:}  There are two different matrices which we require to be PSD.  The first is the moment matrix associated to $\biy$, and is given by 
    \begin{equation*}
        \bM^1_d(\biy) = \ell_{\biy}(\bv_d(p,q,p',q')\bv_d(p,q,p',q')^\top) \succeq 0.
    \end{equation*}
     The second is obtained by setting $t=r=1$ (note that when $t=0$, we recover the moment matrix $\bM^1_d(\biy)$), and is given by 
     \begin{equation*}
         \overline{\cQ}^{1,1}_{d} =  \ell_{\biy}(\bv_d(p,q)\bv_d(p',q')^\top) \succeq 0.
    \end{equation*}
\end{enumerate}

In the first example, we consider computing the GW distance between a pair Gaussian distributions in which the covariance matrices are scalar multiples of each other.  In this simplified setting, the analytical value of the GW distance is known: 
\begin{proposition}~\cite[Proposition 5.2]{Del:2022}
    Suppose $m = n$.  Let $\Sigma_p = P_p D_p P_p^\top$ and $\Sigma_q = P_q D_q P_q^\top$ be the respective spectral decompositions.  Suppose that $\Sigma_p$ is non-singular and that there exists a scalar $\lambda \geq 0$ such that $D_p = \lambda D_q$.  Then 
    \begin{equation*}
        \mathrm{GW}^2(\mu_p,\mu_q) := (\lambda -1)^2 (4\mathrm{tr}(\Sigma_p)^2+ 8\|\Sigma_p\|_{F}^2).
    \end{equation*}
\end{proposition}
Here, and in what follows, $\|\cdot\|_{F}$ denotes the Frobenius norm.

The GW distance is invariant under orthogonal transformations on the input distributions.  We note that our proposed SDP relaxation also obeys the same invariance.  In the following experiments, without loss of generality, we choose $\Sigma_p = I$ and $\Sigma_q = \lambda I$ to be the identity, and we set the mean vectors to be the zero vector. 

\begin{table}[H]
\centering
\begin{tabular}{|c|c|c|c|c|c|}
\hline
$n$ & $\lambda$ & $\#$ Variables & $\#$ Variables & $\operatorname{GW^2}$ & SDP-LB \\
& & (Original) & (Reduced) & (Analytic value) & (Our method) \\
\hline 
2 & 2 & 495 & 255 & 32 & 32 \\
2 & 3 & 495 & 255 & 128 & 128 \\
2 & 4 & 495 & 255 & 288 & 288 \\
2 & 5 & 495 & 255 & 512 & 512 \\
\hline
3 & 2 & 1820 & 924 & 60 & 60 \\
3 & 3 & 1820 & 924 & 240 & 240 \\
3 & 4 & 1820 & 924 & 540 & 540 \\
3 & 5 & 1820 & 924 & 960 & 960 \\
\hline
4 & 2 &4845  & 2445  & 96 & 96 \\
4 & 3 &4845  & 2445  & 384 & 384 \\
4 & 4 &4845  & 2445  & 864 & 864 \\
4 & 5 &4845  & 2445  & 1536 & 1536 \\
\hline
\end{tabular}
\caption{Comparison of the lower bounds of the GW-distance obtained via our method (SDP-LB) versus the analytical values ($\mathrm{GW}^2$).}
\end{table}
In all instances considered, our proposed SDP relaxation attains the same value as the analytical value provided in \cite[Proposition 5.2]{Del:2022}, which suggest that the relaxation is exact in this instance.  %We conjecture that, for this particular class of instances, the SDP relaxation is exact. A rigorous proof of this conjecture is left for future work.

The closed-form value of the GW distance for general pairs of Gaussian distributions is not known, to the best of our knowledge.  Instead, the work in work~\cite{Del:2022} provides an analytical bound.  
%To compare the lower bounds given by these SDP relaxation in general setting, we recall the lower bound $\operatorname{LGW}(\mu_p,\mu_q)$ of the Gromov-Wasserstein between Gaussian distributions introduced in the work~\cite{Del:2022}: 
Suppose without any loss of generality that $m \geq n$. Let $P_p,D_p$ and $P_q,D_q$ be the respective diagonalizations of $\Sigma_p = P_p D_p P_p^\top$ and $\Sigma_q = P_q D_q P_q^\top$ which sort eigenvalues in decreasing order. We suppose that $\Sigma_p$ is non-singular. A lower bound for $\operatorname{GW}(\mu_p,\mu_q)$  \cite{Del:2022} is given by
\begin{multline*}
    \operatorname{LGW}(\mu_p,\mu_q) = 4(\operatorname{tr}(D_p)-\operatorname{tr}(D_q))^2 + 4(\|D_p\|_F - \|D_q\|_F)^2 + 4\|D_p^{(n)}-D_q\|_F \\
    + 4(\|D_p^{(n)}\|_F^2 - \|D_q\|_F^2).
\end{multline*}
We use this analytical lower bound as a benchmark for our SDP relaxation.  We summarize our numerical results in the following table:
\begin{table}[H]
\centering
\begin{tabular}{|c|c|c|c|c|c|}
\hline
Distributions & $d$ & $\#$ Variables & $\#$ Variables & Analytical LB & SDP LB \\
& & (Original) & (Reduced) & ($\operatorname{LGW}^2$) & (Our method)\\
\hline 
{\tiny $\mathcal{N}\left(\begin{pmatrix}
    0 \\
    0
\end{pmatrix}, \begin{pmatrix}
    1 & 0\\
    0 & 1
\end{pmatrix}\right)$}  & 2 & 495 & 255 &  & 74.416 \\
and & 3 & 3003 & 1519 & 75.208 & 74.416\\
{\tiny $\mathcal{N}\left(\begin{pmatrix}
    0 \\
    0
\end{pmatrix}, \begin{pmatrix}
    2 & 0\\
    0 & 3
\end{pmatrix}\right)$} & 4 & 12870 & 6470 &  & 75.064 \\
 & 5 &43758  &21942 &  &77.067\\
\hline

{\tiny $\mathcal{N}\left(\begin{pmatrix}
    0 \\
    0
\end{pmatrix}, \begin{pmatrix}
    1 & 0\\
    0 & 1
\end{pmatrix}\right)$}  & 2 & 495 & 255 &  & 138.807 \\
and & 3 & 3003 & 1519 & 141.04 & 138.807\\
{\tiny $\mathcal{N}\left(\begin{pmatrix}
    0 \\
    0
\end{pmatrix}, \begin{pmatrix}
    2 & 0\\
    0 & 4
\end{pmatrix}\right)$} & 4 & 12870 & 6470 &  & 140.593 \\
 & 5 &43758  &21942 &  &151.840\\
\hline

{\tiny $\mathcal{N}\left(\begin{pmatrix}
    0 \\
    0
\end{pmatrix}, \begin{pmatrix}
    1 & 0\\
    0 & 1
\end{pmatrix}\right)$}  & 2 & 495 & 255 &  & 226.148 \\
and & 3 & 3003 & 1519 & 231.067 & 226.148\\
{\tiny $\mathcal{N}\left(\begin{pmatrix}
    0 \\
    0
\end{pmatrix}, \begin{pmatrix}
    2 & 0\\
    0 & 5
\end{pmatrix}\right)$} & 4 & 12870 & 6470 &  & 229.150 \\
 & 5 &43758  &21942 &  &292.005\\
\hline
\end{tabular}
\end{table}

We observe that, as we increase the relaxation order $d$, the SDP relaxations become progressively tighter and eventually improve upon the analytical lower bound $\operatorname{LGW}^2$.  This behavior is consistently observed across numerous numerical experiments. Here, we present three representative examples illustrating this phenomenon.

\subsection{Gromov-Wasserstein between uniform distributions}

In this section, we run experiments on uniform distributions, whose supports are balls and spheres. Since the supports for distributions are compact, we need to add the PSD condition of localizing matrices to enhance the SDP relaxation mentioned in Section~\ref{sec: strengthening hierarchy}.

Let $B^n(1)$ and $B^n(2)$ be the $n$--dimensional balls of radii $1$ and $2$ respectively. Then the localizing matrix are $\bM^1_{d-1}(g_1\biy) \succeq 0$ and $\bM^1_{d-1}(g_2\biy) \succeq 0$, where $g_1(\bx):= 1 - \|\bx\|^2$ and $g_2(\bz) = 2 - \|\bz\|^2$. An upper bound in this case is calculated as follows: 
Let $\mu=\operatorname{Unif}(B^n(1)),\; \nu=\operatorname{Unif}(B^n(2))$.

Consider the dilation coupling
\begin{displaymath}
    Y=2X,\qquad X\sim \mu .
\end{displaymath}
Then $Y\sim \nu$. Let $((X,Y))$ and $((X',Y')$ be two independent copies under this coupling. Since $Y-Y'=2(X-X')$, we have $\|Y-Y'\|^2=4\|X-X'\|^2$. Therefore,

\begin{displaymath}
\left(\|X-X'\|^2-\|Y-Y'\|^2\right)^2 = \left(\|X-X'\|^2-4\|X-X'\|^2\right)^2 =
9\|X-X'\|^4. 
\end{displaymath}
Hence this coupling gives

\begin{displaymath}
    \operatorname{GW}_{2,2}^2(\mu_p,\mu_q) \leq 9\,\mathbb{E}\|X-X'\|^4,
\end{displaymath}
where $X,X'\sim \operatorname{Unif}(B_n(0,1))$ are independent. It remains to compute \(\mathbb{E}\|X-X'\|^4\). We have
\begin{displaymath}
    \|X-X'\|^2 = \|X\|^2+\|X'\|^2-2\langle X,X'\rangle .
\end{displaymath}
Using independence and symmetry, we obtain

\begin{displaymath}
    \mathbb{E}\|X-X'\|^4 = 2\mathbb{E}\|X\|^4 +  2\left(\mathbb{E}\|X\|^2\right)^2
+ 4\mathbb{E}\langle X,X'\rangle^2 .
\end{displaymath}
For $X\sim \operatorname{Unif}(B^n(1))$,
\begin{displaymath}
    \mathbb{E}\|X\|^2=\frac{n}{n+2},\qquad \mathbb{E}\|X\|^4=\frac{n}{n+4}.
\end{displaymath}
Moreover, by isotropy,

\begin{displaymath}
    \mathbb{E}[XX^\top]=\frac{1}{n+2}I_n.
\end{displaymath}
Thus

\begin{displaymath}
    \mathbb{E}\langle X,X'\rangle^2 = \operatorname{tr} \left( \mathbb{E}[XX^\top]\mathbb{E}[X'X'^\top] \right) = \operatorname{tr} \left(\frac{1}{(n+2)^2}I_n \right) = \frac{n}{(n+2)^2}.
\end{displaymath}

Consequently,

\begin{displaymath}
    \mathbb{E}\|X-X'\|^4 = \frac{2n}{n+4} + \frac{2n^2}{(n+2)^2} + 
\frac{4n}{(n+2)^2}.
\end{displaymath}
Simplifying gives

\begin{displaymath}
    \mathbb{E}\|X-X'\|^4 = \frac{4n(n+3)}{(n+2)(n+4)}.
\end{displaymath}
Therefore,
\begin{displaymath}
    \operatorname{GW}_{2,2}^2(\mu_p,\mu_q) \leq  \frac{36n(n+3)}{(n+2)(n+4)}.
\end{displaymath}
We use this upper bound as a benchmark to exanimate the tightness of our experiments, which is summarized in the following table. 
\begin{table}[H]
\centering
\begin{tabular}{|c|c|c|c|c|c|}
\hline
$n$ & $d$ & $\#$ Variables & $\#$ Variables & Upper bound & Lower bound\\
& & (Original) & (Reduced) &  & \\
\hline 
2 & 2 & 495 & 255 & 15 & 15 \\
3 & 2 & 1820 & 924 & 18.514 & 18.514\\
4 & 2 & 4845 & 2445 & 21 & 21 \\
5 & 2 & 10626 &5346 & 22.857 & 22.857 \\
\hline
\end{tabular}
\end{table}

Let $S^{n-1}(1)$ and $S^{n-1}(2)$ be the $n$--dimensional spheres of radii $1$ and $2$ respectively. Then the localizing matrix are $\bM^1_{d-1}(h_1\biy) = 0$ and $\bM^1_{d-1}(h_2\biy) = 0$, where $h_1(\bx):= 1 - \|\bx\|^2$ and $h_2(\bz) = 2 - \|\bz\|^2$. An upper bound in this case is calculated as follows: 
Let $\mu=\operatorname{Unif}(S^{n-1}(1)),\; \nu=\operatorname{Unif}(S^{n-1}(2))$.
Consider the dilation coupling
\begin{displaymath}
    Y=2X,\qquad X\sim \mu \quad \Rightarrow \quad Y\sim \nu
\end{displaymath}
We calculate the upper bound based on this dilation coupling similarly to the case of balls to obtain that 
\begin{displaymath}
    \operatorname{GW}(\mu_p,\mu_q) \leq \frac{36(n+1)}{n}.
\end{displaymath}
We also use this upper bound as a benchmark to exanimate the tightness of our experiments, summarized in the following table. 
\begin{table}[H]
\centering
\begin{tabular}{|c|c|c|c|c|c|}
\hline
$n$ & $d$ & $\#$ Variables & $\#$ Variables & Upper bound & Lower bound\\
& & (Original) & (Reduced) &  & \\
\hline 
2 & 2 & 495 & 255 & 54 & 54 \\
3 & 2 & 1820 & 924 & 48 & 48\\
4 & 2 & 4845 & 2445 & 45 & 45 \\
5 & 2 & 10626 &5346 & 43.2 & 43.2 \\
\hline
\end{tabular}
\end{table}

\subsection{Polynomial kernel cost problem}

In this example, we consider solving the following problem:
\begin{equation} \label{eq:exp3_mainproblem}
    \inf_{\pi \in \Pi(\mu_p,\mu_q)} \int c(p_1,q_1)~d\pi + \int K(p_1,q_1,p_2,q_2)~d \pi \otimes \pi,
\end{equation}
where $\mu_{p} \sim \mathcal{N}(\mathbf{m}_{p},\Sigma_{p})$ and $\mu_{q} \sim \mathcal{N}(\mathbf{m}_{q},\Sigma_{q})$ are Gaussian distributions in $\mathbb{R}^{m}$ and $\mathbb{R}^{n}$ respectively.  The vector of monomials of degree up to $2$ for the variables $(p_1,q_1)$ and $(p_2,q_2)$ are denoted by $\bv_2(p_1,q_1)$ and $\bv_2(p_2,q_2)$, respectively.  The polynomial $c$ is defined by 
\begin{equation*}
    c(p_1,q_1)= \langle \bc, \bv_2(p_1,q_1) \rangle,
\end{equation*}
and the polynomial kernel $K$ is defined by 
\begin{equation*}
    K(p_1,q_1,p_2,q_2) = \bv_2(p_1,q_1)' * Q * \bv_2(p_2,q_2).
\end{equation*}
The coefficient vector $\bc$ and the coefficient matrix $Q$ are randomly generated, and we specify the generation process later.  We solve this problem via our proposed framework with the choice of parameters $r = 1$ and $d = 2$.  The resulting SDP relaxation is 
\begin{equation} \label{eq:exp3_relax}
    \inf_{\biy \in \cM} \ell_{\biy} \left(c(p_1,q_1) + K(p_1,q_1,p_2,q_2)\right),
\end{equation}
where $\cM$ is the collection of (truncated) pseudo-moment sequences of degree up to $4$ satisfying the following conditions:
\begin{enumerate}
    \item[(i)] [Symmetry]: $\biy_{(\alpha,\beta)} = \biy_{(\beta,\alpha)} \quad \forall \alpha, \beta \in \NN^{m+n};\; |\alpha| + |\beta| \leq 4$.
    \item[(ii)] [Marginal condition]: 
    \begin{multline*}
        y_{(\alpha,0,\gamma)} = \left(\int p_1^\alpha~d\mu\right)\cdot y_{(0,0,\gamma)},\quad \mbox{and} \quad y_{(0,\beta,\gamma)} = \left(\int p_1^\alpha~d\nu\right)\cdot y_{(0,0,\gamma)}\\
        \quad \forall \alpha \in \NN^m,\; \beta \in \NN^n,\; \gamma \in \NN^{m+n}\; |\alpha| + |\gamma| \leq 4, |\beta| + |\gamma| \leq 4.
    \end{multline*}
    \item[(iii)] [PSD]: 
    \begin{equation*}
    \begin{aligned}
        \bM_2(\biy) ~=~ & \ell_{\biy} \left( \bv_2(p_1,q_1,p_2,q_2)\bv_2(p_1,q_1,p_2,q_2)^\top \right) \succeq 0, \\
        q(\biy) ~=~ & \ell_{\biy} \left( \bv_2(p_1,q_1)\bv_2(p_2,q_2)^\top \right) \succeq 0.
    \end{aligned}
    \end{equation*}
\end{enumerate}
The experiments are conducted as follows: For each choice of dimensions $(m,n)$, we run 100 trials with random instantiations of the parameters $\mathbf{m}_p \in \RR^m,\;\Sigma_p \in S_+^m,\; \mathbf{m}_q \in \RR^n,\; \Sigma_p \in S_+^n,\; \bc \in \RR^{\binom{m+n+2}{2}}$ and $Q \in S_+^{\binom{m+n+2}{2}}$.  Here, $\bc$ is a random vector whose entries are drawn from $\mathrm{Unif}[0,1]$, while $Q = GG^\top$, where $G$ is a random matrix whose entries are drawn standard Gaussian distribution. 

We check if solution from the SDP relaxation \eqref{eq:exp3_relax} is exact using the following steps:  Let $\hat{\biy}$ be an optimal solution to \eqref{eq:exp3_relax}.  Consider 
\begin{equation*}
    \mathrm{Proj}(\hat{\biy}) := \ell_{\hat{\biy}}(\bv_2(p_1,q_1)).    
\end{equation*}
Because $\mathrm{Proj}(\hat{\biy})$ retains all moments with degree at most $2$, we can extract the mean vector $\mathbf{m}$ and covariance matrix $\Sigma$.  Furthermore, by the marginal conditions, we have
\begin{displaymath}
    \mathbf{m} = (\mathbf{m}_p,\mathbf{m}_q)^\top,\quad \mbox{and} \quad \Sigma = \begin{pmatrix}
        \Sigma_p & * \\
        * & \Sigma_q
    \end{pmatrix}.
\end{displaymath}
We are guaranteed that $\Sigma$ is PSD because it is a principle submatrix of $\bM_2(\hat{\biy})$.  Therefore, $\bm$ and $\Sigma$ are valid mean and covariance matrix the Gaussian distribution $\mathcal{N}(\mathbf{m},\Sigma) \in \Pi(\mu_{p},\mu_{q})$. In particular, the following specifies a valid upper bound to \eqref{eq:exp3_mainproblem}.
\begin{equation*}
    \operatorname{UB} := \langle \bc, \mathrm{Proj}(\hat{\biy}) \rangle + \mathrm{Proj}(\hat{\biy})^\top*Q*\mathrm{Proj}(\hat{\biy}).
\end{equation*}
On the other hand, the solution of the SDP relaxation gives a valid lower bound to \eqref{eq:exp3_mainproblem}:
\begin{equation*}
    \operatorname{LB} := \ell_{\hat{\biy}} \left(c(p_1,q_1) + K(p_1,q_1,p_2,q_2)\right).
\end{equation*}
Combining these, we declare the relaxation to be exact if $\operatorname{UB}/\operatorname{LB} \in [0.99,1.01]$.  We summarize the results in the following table: 
\begin{table}[H]
\centering
\begin{tabular}{|c|c|c|c|c|}
\hline
$(m,n)$ & Random & $\#$ Variables & $\#$ Variables & Fraction of \\
& model of $Q$ & (Original) & (Reduced) & Exact Instances \\
\hline 
(2,3) & Random Sym. & 1001 & 511 & 99/100 \\
(2,4) & Random Sym. & 1820 & 924 &  99/100 \\
(3,3) & Random Sym. & 1820 & 924 & 99/100 \\
(2,3) & Random Neg Def. & 1001 & 511 & 77/100 \\
(2,4) & Random Neg Def. & 1820 & 924 &  92/100 \\
(3,3) & Random Neg Def. & 1820 & 924 & 85/100 \\
\hline
\end{tabular}
\end{table}

% \begin{enumerate}
%     \item $m= n =3$, the number of variables is 1820 and the symmetry reduces the number variables down to 924, and 99/100 trials are tight.
%     \item $m=2,\; n=4$ , the number of variables is 1820 and the symmetry reduces the number variables down to 924, and 99/100 trials are tight.
%     \item $m=2,\; n=3$ , the number of variables is 1001 and the symmetry reduces the number variables down to 511, and 99/100 trials are tight.
%     \item $m=2,\; n=3$ , and $Q$ is negative semidefintie. the number of variables is 1001 and the symmetry reduces the number variables down to 511, and 85/100 trials are tight.
%     \item $m= n =3$, and $Q$ is negative semidefintie. the number of variables is 1820 and the symmetry reduces the number variables down to 924, and 92/100 trials are tight.
%     \item $m=2,\; n=4$, and $Q$ is negative semidefintie. the number of variables is 1820 and the symmetry reduces the number variables down to 924, and 77/100 trials are tight.
% \end{enumerate}

\bibliographystyle{plain}
\bibliography{quadmeas}

\begin{thebibliography}{10}

\bibitem{BPT:12}
Grigoriy Blekherman, Pablo~A. Parrilo, and Rekha~R. Thomas.
\newblock {\em Semidefinite {O}ptimization and {C}onvex {A}lgebraic {G}eometry}.
\newblock MOS-SIAM Series on Optimization. SIAM, 2012.

\bibitem{BdK:02}
Immanuel~M. Bomze and Etienne de~Klerk.
\newblock Solving {S}tandard {Q}uadratic {O}ptimization {P}roblems via {L}inear, {S}emidefinite and {C}opositive {P}rogramming.
\newblock {\em Journal of Global Optimization}, 24, 2002.

\bibitem{chen:2023semidefinite}
Junyu Chen, Binh~Tuan Nguyen, Shang~Hui Koh, and Yong~Sheng Soh.
\newblock Semidefinite {R}elaxations of the {G}romov-{W}asserstein {D}istance.
\newblock In {\em NeurIPS}, 2024.

\bibitem{deFinetti:37}
Bruno de~Finetti.
\newblock La pr{\'e}vision: ses lois logiques, ses sources subjectives.
\newblock {\em Annales de l'Institut Henri Poincar{\'e}}, 7(1):1--68, 1937.

\bibitem{De:2003}
Marcel de~Jeu.
\newblock Determinate {M}ultidimensional {M}easures, the {E}xtended {C}arleman {T}heorem and {Q}uasi-analytic {W}eights.
\newblock {\em The Annals of Probability}, 31(3):1205--1227, 2003.

\bibitem{Del:2022}
Julie Delon, Agnes Desolneux, and Antoine Salmona.
\newblock Gromov--{W}asserstein distances between {G}aussian distributions.
\newblock {\em Journal of Applied Probability}, 59(4):1178--1198, 2022.

\bibitem{HewittSavage:55}
Edwin Hewitt and Leonard~J. Savage.
\newblock Symmetric {M}easures on {C}artesian {P}roducts.
\newblock {\em Transactions of the American Mathematical Society}, 80(2):470--501, 1955.

\bibitem{Las:01}
Jean~B. Lasserre.
\newblock Global {O}ptimization with {P}olynomials and the {P}roblem of {M}oments.
\newblock {\em SIAM Journal on Optimization}, 11(3):796--817, 2001.

\bibitem{Las:09}
Jean~B. Lasserre.
\newblock {\em Moments, {P}ositive {P}olynomials and their {A}pplications}, volume~1.
\newblock World Scientific, 2009.

\bibitem{Lof:2004}
Johan L{\"o}fberg.
\newblock {YALMIP}: A toolbox for modeling and optimization in {MATLAB}.
\newblock In {\em Proceedings of the 2004 IEEE International Symposium on Computer Aided Control Systems Design}, pages 284--289, Taipei, Taiwan, 2004.

\bibitem{Mem:07}
Facundo M{\'e}moli.
\newblock On the use of {G}romov-{H}ausdorff {D}istances for {S}hape {C}omparison.
\newblock In {\em Symposium on Point Based Graphics 07}, 2007.

\bibitem{Mem:11}
Facundo M{\'e}moli.
\newblock Gromov--{W}asserstein {D}istances and the {M}etric {A}pproach to {O}bject {M}atching.
\newblock {\em Foundations of Computational Mathematics}, 11:417--487, 2011.

\bibitem{Mosek:2026}
{MOSEK ApS}.
\newblock {\em The {MOSEK} Optimization Toolbox for {MATLAB} Manual}, 2026.
\newblock Version 11.2.

\bibitem{Motzkin:65}
Theodore~S. Motzkin and Ernst~G. Straus.
\newblock Maxima for {G}raphs and a {N}ew {P}roof of a {T}heorem of {Tur{\'a}n}.
\newblock {\em Canadian Journal of Mathematics}, 17:533--540, 1965.

\bibitem{MN:24}
Olga Mula and Anthony Nouy.
\newblock Moment-{SoS} {M}ethods for {O}ptimal {T}ransport {P}roblems.
\newblock {\em Numer. Math.}, 156(4):1541--1578, 2024.

\bibitem{Nie:2013Approximation}
Jiawang Nie.
\newblock An {A}pproximation {B}ound {A}nalysis for {Lasserre's} {R}elaxation in {M}ultivariate {P}olynomial {O}ptimization.
\newblock {\em Journal of the Operations Research Society of China}, 1:313--332, 2013.

\bibitem{Nie:2014}
Jiawang Nie.
\newblock Optimality {C}onditions and {F}inite {C}onvergence of {Lasserre's} hierarchy.
\newblock {\em Mathematical Programming}, 146(1--2):97--121, 2014.

\bibitem{Par:00}
Pablo~A. Parrilo.
\newblock {\em Structured {S}emidefinite {P}rograms and {S}emialgebraic {G}eometry {M}ethods in {R}obustness and {O}ptimization}.
\newblock Phd thesis, California Institute of Technology, 2000.

\bibitem{Pass:15}
Brendan Pass.
\newblock Multi-marginal optimal transport: {T}heory and applications.
\newblock {\em ESAIM: Mathematical Modelling and Numerical Analysis}, 49(6):1771--1790, 2015.

\bibitem{Put:93}
Mihai Putinar.
\newblock Positive polynomials on compact semi-algebraic sets.
\newblock {\em Indiana University Mathematics Journal}, 42(3):969--984, 1993.

\bibitem{San:15}
Filippo Santambrogio.
\newblock {\em Optimal {T}ransport for {A}pplied {M}athematicians: Calculus of Variations, PDEs, and Modeling}.
\newblock Birkh{\"a}user, 2015.

\bibitem{Schmud:2017}
Konrad Schm{\"u}dgen.
\newblock {\em The {M}oment {P}roblem}.
\newblock Springer, 2017.

\bibitem{slot:2022}
Lucas Slot.
\newblock Sum-of-squares hierarchies for polynomial optimization and the {C}hristoffel--{D}arboux kernel.
\newblock {\em SIAM Journal on Optimization}, 32(4):2612--2635, 2022.

\bibitem{Sturm:06}
Karl-Theodor Sturm.
\newblock On the {G}eometry of {M}etric {M}easure {S}paces.
\newblock {\em Acta Mathematica}, 196:65--131, 2006.

\bibitem{tran:2026sum}
Hoang~Anh Tran, Binh~Tuan Nguyen, and Yong~Sheng Soh.
\newblock Sum-of-{S}quares {H}ierarchy for the {G}romov--{W}asserstein {P}roblem.
\newblock {\em SIAM Journal on Optimization}, 36(2):729--759, 2026.

\bibitem{tran:2025convergence}
Hoang~Anh Tran and Kim-Chuan Toh.
\newblock On the convergence rates of moment-{SOS} {H}ierarchies {A}pproximation of {T}runcated {M}oment {S}equences.
\newblock {\em arXiv preprint arXiv:2507.00572}, 2025.

\bibitem{tran:2026convergence}
Hoang~Anh Tran and Kim-Chuan Toh.
\newblock Convergence {R}ates of {S}um-of-{S}quares {H}ierarchies for {P}olynomial {S}emidefinite {P}rograms.
\newblock {\em SIAM Journal on Optimization}, 36(2):760--790, 2026.

\bibitem{Vil:03}
Cédric Villani.
\newblock {\em Topics in {O}ptimal {T}ransportation}.
\newblock American Mathematical Society, 2003.

\bibitem{Vil:08}
Cédric Villani.
\newblock {\em Optimal {T}ransport: {O}ld and {N}ew}.
\newblock Springer Berlin Heidelberg, 2008.

\bibitem{Vil:2016}
Soledad Villar, Afonso~S Bandeira, Andrew~J Blumberg, and Rachel Ward.
\newblock A {P}olynomial-time {R}elaxation of the {G}romov-{H}ausdorff distance.
\newblock {\em arXiv preprint arXiv:1610.05214}, 2016.

\end{thebibliography}

\end{document}